\documentclass{article} 

\usepackage[utf8]{inputenc} 
\usepackage[english]{babel}

\usepackage{geometry} 
\usepackage{graphicx} 
\usepackage{amsmath}
\usepackage{amssymb}

\usepackage{amsthm}

\usepackage{algorithm}
\usepackage{algpseudocode}
\usepackage{hyperref}

\usepackage{booktabs} 
\usepackage{array} 
\usepackage{paralist} 
\usepackage{verbatim} 
\usepackage{subfig} 
\usepackage{amsfonts}
\usepackage{fancyhdr} 
\usepackage{sectsty}
\allsectionsfont{\sffamily\mdseries\upshape} 

\usepackage[nottoc,notlof,notlot]{tocbibind} 
\usepackage[titles,subfigure]{tocloft} 

\usepackage{mathrsfs}
\usepackage[all,cmtip]{xy}

\newcommand{\ZZ}{\mathbb{Z}}
\newcommand{\RR}{\mathbb{R}}
\newcommand{\QQ}{\mathbb{Q}}
\newcommand{\CC}{\mathbb{C}}

\newcommand{\PP}{\mathbb{P}}

\newcommand{\OO}{\mathcal{O}}

\newtheorem{question}{Question}
\newtheorem*{question*}{Question}
\newcommand{\spec}{\text{Spec}\hspace{0.5mm}
}

\newcommand{\McS}{\mathcal{S}}

\newcommand{\ra}{\rightarrow}

\newcommand{\Alb}{\textnormal{Alb}}

\newcommand{\McH}{\mathcal{H}}

\newcommand{\Pic}{\textnormal{Pic}}
\newcommand{\Nef}{\textnormal{Nef}}

\newcommand{\Endo}{\textnormal{End}}
\newcommand{\NE}{\overline{\textnormal{NE}}}

\newcommand{\Num}{\equiv_\textnormal{num}}
\newcommand{\Lin}{\equiv_\textnormal{lin}}

\newcommand{\pp}{\prime}

\newcommand{\Qb}{\overline{\mathbb{Q}}}
\newcommand{\Send}{\textnormal{SEnd}}
\newcommand{\Iamp}{\textnormal{IAmp}}

\newtheorem{theorem}{Theorem}[section]
\newtheorem*{theorem*}{Theorem}
\newtheorem*{conjecture*}{Conjecture}
\newtheorem{conjecture}{Conjecture}
\newtheorem*{proposition*}{Proposition}

\newtheorem{corollary}{Corollary}[theorem]
\newtheorem{lemma}[theorem]{Lemma}
\newtheorem*{lemma*}{Lemma}
\newtheorem{definition}[theorem]{Definition}
\newtheorem{proposition}[theorem]{Proposition}

\newtheorem{example}{Example}
\newtheorem*{example*}{Example}
\newtheorem{remark}[theorem]{Remark}
\newtheorem*{remark*}{Remark}

\newtheorem*{Ingredients*}{Ingredients of the minimal model program for $K_X$}

\newtheorem{Case Q}{Case Q}

\usepackage{color}
\newif\ifhascomments \hascommentstrue
\ifhascomments
  \newcommand{\matt}[1]{{\color{red}[[\ensuremath{\spadesuit\spadesuit\spadesuit} #1]]}}
  \newcommand{\brett}[1]{{\color{blue}[[\ensuremath{\heartsuit\heartsuit\heartsuit} #1]]}}
\else
  \newcommand{\matt}[1]{}
  \newcommand{\brett}[1]{}
\fi

\author{Brett Nasserden\\ Department of Mathematics, Western University\\ \href{mailto:bnasserd@uwo.ca}{bnasserd@uwo.ca} }
\title{On the realizability of arithmetic degrees of morphisms}
\date{}

\begin{document}

\maketitle
\begin{abstract}
The Kawaguchi-Silverman conjecture relates two different invariants of a surjective endomorphism, the dynamical and arithmetic degrees. As the Kawaguchi-Silverman conjecture is only meaningful when a morphism has a Zariski dense orbit, it has no content for varieties with positive Kodaira dimension. A generalization of the Kawaguchi-Silverman conjecture which is meaningful in positive Kodaira dimension is the so called sAND conjecture, which involves the set of "small" arithmetic degrees. Kawaguchi and Silverman showed that a small arithmetic degree is the modulus of an eigenvalue of $f^*\colon N^1(X)\rightarrow N^1(X)$. In this article we investigate which possible eigenvalues arise as an arithmetic degree. We show that surjective endomorphisms of abelian varieties may have eigenvalues which are not arithmetic degrees. Conversely, we show that every eigenvalue of a surjective endomorphism of a toric variety is an arithmetic degree using the minimal model program. Finally, we investigate how the minimal model program may be applied to study this realizability question for varieties that admit an int-amplified endomorphism.        
\end{abstract}

\section{Introduction}
Let $X$ be a normal projective variety defined over $\overline{\QQ}$ and $f\colon X\ra X$ a surjective endomorphism. In this article we study the dynamics of $f$ motivated by the Kawaguchi-Silverman conjecture. The premier notion of complexity of a surjective endomorphism on an algebraic variety is its dynamical degree. 
\begin{equation}
		\lambda_1(f)=\lim_{n\ra \infty}((f^n)^*H \cdot H^{\dim X-1})^{\frac{1}{n}},\end{equation}
This limit exists and is independent of $H$ and the product is the intersection product. Alternatively we may define $\lambda_1(f)$ as the \emph{spectral radius} of the pull-back function $f^*\colon N^1(X)\ra N^1(X)$ on the Neron-Severi space of $X$. See \cite{MR2753603, Bellon1999,MR1760844,MR4048444} for earlier uses and applications.

From an arithmetic perspective, there is a natural notion of dynamical complexity defined by Kawaguchi and Silverman. Fix a point $P
\in X(\overline{\QQ})$. For simplicity, let everything be defined over $\QQ$ and choose a very ample divisor $H$ $X$ defined by a closed embedding $\phi_H\colon X\ra \PP^N$. For a point $x=[x_0:\ldots:x_N]\in \PP^N$ define $h_{\PP^N}(x)=\log(\max_{0\leq i\leq N}\{\vert x_i\vert\})$. Then the function $h_H(P)=h_{\PP^N}(\phi_H(P)$ gives a measure of the \emph{arithmetic} complexity of $P$ and is called the height of $P$ with respect to $H$. To measure the arithmetic complexity of $P$ with respect to $f$ we measure how the height of $P$ as we apply $f$ to $P$. This gives rise to the notion of the arithmetic degree of a point as defined by Kawaguchi and Silverman.
\begin{equation}
		\alpha_f(P)=\lim_{n\ra\infty}h_H^+(f^n(P))^\frac{1}{n},
	\end{equation}
and $h_H^+=\max\{1,h_H(P)\}$. This limit exists and is independent of the choice of $H$.

The Kawaguchi-Silverman conjecture is a type of Ergodic theorem in arithmetic dynamics that predicts a relationship between these two invariants. 
	\begin{conjecture}[\cite{MR4080251}]\label{conj:KS}
		Let $X$ be a normal projective variety defined over $\bar{\QQ}$ and let $f\colon X\ra X$ be a surjective endomorphism. Let $P\in X(\bar{\QQ})$ and suppose that the orbit $\OO_f(P):=\{P,f(P),f^2(P),...\}$ is Zariski dense in $X$. Then $\alpha_f(P)=\lambda_1(f)$. 
	\end{conjecture}
The conjecture states that when a point $P$ has a complex forward orbit, its arithmetic complexity is the same as the geometric complexity. See \cite{MR4080251,MR4068299,MR3189467,MR3871505,1802.07388} for some recent work in this area.

The requirement of a dense forward orbit makes the conjecture interesting only when $\kappa(X)\leq 0$. When $\kappa(X)>0$ the existence of the Iitaka fibration over a positive dimensional base makes the existence of a dense forward orbit impossible. In light of this the following conjecture was proposed, which is interesting for all Kodaira dimensions and implies Conjecture (\ref{conj:KS}).

	\begin{conjecture}[The sAND conjecture {\cite[Conjecture 1.4]{2002.10976}}]\label{conj:sAND}
		Let $X$ be a normal projective variety and $f\colon X\ra X$ a surjective endomorphism. Let
		\[\McS(X,f,N)=\{P\in X(K):[K:\QQ]\leq N,\alpha_f(P)<\lambda_1(f)\}\]
		Then  $\McS(X,f,N)$ is not Zariski dense in $X$. 
	\end{conjecture}
Here sAND means small arithmetic non-density. It is easy to show that the sAND conjecture implies the Kawaguchi-Silverman conjecture.

The sAND conjecture predicts that all the small arithmetic degree are non-dense. In this article we study the set of possible values for the function $\alpha_f(P)$. In \cite{MR4068299} it was shown that $\alpha_f(P)=\vert \mu\vert$ where $\mu$ is some eigenvalue of $f^*\colon N^1(X)\ra N^1(X)$. We introduce the following definition. \begin{definition}\label{def:potdeg}
	Let $X$ be a normal projective variety defined over $\overline{\QQ}$ and let $f\colon X\ra X$ be a surjective endomorphism and let $\mu$ be an eigenvalue of $f^*\colon N^1(X)\ra N^1(X)$ with $\vert \mu\vert > 1$. 
	
	\begin{enumerate}
	    \item We call such an eigenvalue a \textbf{potential arithmetic degree} or \textbf{potentially arithmetic}.
	    \item  If there is a point $P\in X(\overline{\QQ})$ with $\alpha_f(P)=\vert \mu\vert $ then we call $\mu$ an \textbf{arithmetic degree} or say $\mu$ is \textbf{realizable as an arithmetic degree} or that $\mu$ is \textbf{arithmetic}
	\end{enumerate}. 
	If every potentially arithmetic eigenvalue is arithmetic, then we say that $f$ has \textbf{arithmetic eigenvalues}.
\end{definition}
We choose to exclude the points with $\alpha_f(P)=1$ because this typically occurs even when $1$ is not an eigenvalue for the action of $f^*$ on $\Nef(X)_\RR$. We note that the set of points with $\alpha_f(P)=1$ is extremely interesting from the point of view of the Kawaguchi-Silverman conjecture and deserves further study. We ask the following question.
\begin{question}[Main question]\label{ques:ques2} 
	Let $X$ be a normal projective variety defined over $\overline{\QQ}$ and let $f\colon X\ra X$ be a surjective endomorphism. Is every potential arithmetic degree of $f$ realizable as an arithmetic degree? In other words, is every eigenvalue $\mu$ of $f^*$ with $\vert \mu\vert>1$ arithmetic? 
\end{question}

\begin{subsection}{Our results:}
We first answer question \ref{ques:ques2}.

\begin{theorem}[\ref{thm:MainCounterexample}]
		For each $g\in 2\ZZ_{>0}$ there is an abelian variety $A$  of dimension $g$ defined over a number field $K$ with $\rho(A)=3$ equipped with a surjective endomorphism $f\colon A\ra A$ that has the following properties.
		\begin{enumerate}
			\item $f^*\colon N^1(A)_\QQ\rightarrow N^1(A)_\QQ$ has eigenvalues $a^2>ab>b^2>0$ for some $a,b,c\in \mathbb{Z}$.
			\item $\alpha_f(P)=a^2$ for all $P\notin A(\overline{K})_{\textnormal{tors}}$. In particular, $\alpha_f(P)\in \{1,a^2\}$.
			\item The eigendivisors of $a^2,b^2$ are nef while the eigendivisor of $ab$ is not. 
		\end{enumerate}
	\end{theorem}	
	
We also can show the following.

\begin{theorem}[\ref{cor:Kodiradimrealizability}]
	For any integer $d>1$ there is a smooth projective variety $X$ with $\dim X=d$ such that there is a surjective endomorphism $f\colon X\ra X$ with $\lambda_1(f)>1$ and $f$ has does not have arithmetic eigenvalues. If $d\geq 3$ and $\kappa\in \{-\infty,0,1,\dots , d-2\}$ then $X$ may be chosen with $\kappa(X)=\kappa$. 
		
\end{theorem}
These examples are intimately related to the dynamics of Abelian varieties. If we demand that $X$ does not admit a non-trivial morphism to an abelian variety, we get more positive results. We first handle the case of \emph{equivariant morphisms} of toric varieties. To do this we introduce the following definition in analogy with Abelian varieties. 

\begin{definition}[\ref{def:simpletoricvar}]
 Let $X_\Sigma$ be a $\QQ$-factorial projective toric variety defined over $\Qb$. We say that $X_\Sigma$ is decomposable if
		\[X_\Sigma=X_{\triangle_1}\times X_{\triangle_2}\]
with each $X_{\triangle_i}$ a $\QQ$-factorial projective toric variety of dimension at least $1$. We say that $X_\Sigma$ is simple if it is not decomposable.
\end{definition}

We also have a notion of a  \emph{dynamically simple} toric variety.
\begin{definition}[\ref{def:linearlysimple}]
  Let $X_\Sigma$ be a  projective toric variety defined over $\Qb$. We say that $X_\Sigma$ is \textbf{linearly simple} if every surjective toric  morphism is induced by a homomorphism of tori $(x_1,\dots, x_n)\mapsto (x_1^d,\dots, x_n^d)$ for some $d>0$ after possibly iterating the morphism. 
\end{definition}

We then show that these two notions are the same.

\begin{theorem}[\ref{thm:simple=linsimple}]
Let $X_\Sigma$ be a $\QQ$-factorial projective toric variety defined over $\Qb$. Then $X_\Sigma$ is linearly simple if and only if $X_\Sigma$ is simple.
	   
\end{theorem}

This result can then be leveraged to give new proofs of the following results.

\begin{theorem}
Let $X_\Sigma$ be a $\QQ$-factorial toric variety defined over $\Qb$. Let $f\colon X_\Sigma\ra X_\Sigma $ be an equivariant surjective endomorphism. 

\begin{enumerate}
    \item  The sAND conjecture \ref{conj:sAND} is true for $f$.
    \item The morphism $f$ has arithmetic eigenvalues.
\end{enumerate}
\end{theorem}

Finally, we use the minimal model program to prove a similar result for all surjective (not necessarily equivariant) endomorphisms of $\QQ$-factorial toric varieties.

\begin{theorem}[\ref{thm:toricrealworks}]\label{introtoricthm}
	Let $X$ be $\QQ$-factorial toric variety defined over $\overline{\QQ}$. Let $f\colon X\ra X$ be a surjective endomorphism. Then $f$ has arithmetic eigenvalues. 
\end{theorem}
We then turn to generalizing our methods using the minimal program with toric varieties to certain varieties admitting int-amplified endomorphisms. When trying to study endomorphisms of higher dimensional projective algebraic varieties, it is crucial to be able to apply the methods of higher dimensional algebraic geometry, namely the minimal model program. Zhang and collaborators have begun a program where one studies endomorphisms from this perspective, see \cite{MR3742591,MR4117085}. The varieties which their methods work best are those which admit \emph{at least one} int-amplified endomorphism. The existence of this single int-amplified endomorphism has profound affects on the birational geometry of the underlying variety, which simplifies the study of \emph{all} surjective endomorphisms.
\begin{theorem}
	Let $X$ be a normal $\QQ$-factorial surface with at worst terminal singularities and finitely generated nef cone that is rationally connected over $\overline{\QQ}$ and $\Alb(X)=0$ that admits an int-amplified endomorphism. Let $f\colon X\ra X$ be a surjective endomorphism and $\lambda$ a potential arithmetic degree of $f$. Then $\lambda$ is realizable as an arithmetic degree.  
\end{theorem}

\end{subsection}

\begin{section}{Preliminaries}\label{sec:surjectivemorphisms}
One of our main strategies when dealing with toric varieties and more generally varieties admitting an int-amplified endomorphism will be to apply the minimal model program. To do this efficiently one needs descend morphisms along Mori-fiber spaces and birational contractions. The following is when dealing with Mori-fiber spaces and does not require any additional assumptions on $X$.
\begin{lemma}[{\cite[Lemma 6.2]{1802.07388}}]\label{lemma:iterationlemma}
		Let $\pi\colon X\ra Y$ be a Mori-fiber space. Suppose that $f\colon X\ra X$ is a surjective endomorphism. Then there is some iterate $f^n\colon X\ra X$ and $g\colon Y\ra Y$ such that
		\[\xymatrix{X\ar[r]^{f^n}\ar[d]_\pi & X\ar[d]^\pi\\ Y\ar[r]_g & Y}\]
		commutes. 
\end{lemma}
When one is dealing with the birational contractions that appear in the minimal model program there is greater difficulty. Suppose that $f\colon X\ra X$ is a surjective endomorphism of projective varieties. Let $\pi\colon X\ra Y$ be a birational contraction morphism. We obtain dominant rational mapping $f\colon f\colon Y\dashrightarrow Y$, however, this rational map may not extend to a \emph{morphism} $g\colon Y\ra Y$. To ensure that this occurs, Zhang and Meng introduced the following notion.
\begin{definition}[Int amplified, amplified and polarized endomorphisms]
			Let $X$ be a projective variety defined over $\Qb$ and $f\colon X\ra X$ a surjective endomorphism. We say that $f$ is a \emph{polarized} endomorphism if there is some ample $\QQ$-Cartier divisor $L$ on $X$ such that $f^*L\Lin qL$ for some $q>1$. We say that $f$ is \emph{amplified} if there is a $\QQ$-Cartier divisor $L$ such that $f^*L-L$ is ample. We say $f$ is an  \emph{int-amplified} endomorphism if there is some ample $\QQ$-Cartier divisor $L$ with $f^*L-L$ being ample.
		\end{definition}
		
		Int amplified endomorphisms can be characterized in terms of their eigenvalues being large.

		\begin{proposition}[Eigenvalues determine int-amplified endomorphisms: Theorem 3.3 \cite{MR4074056}]\label{thm:intampeigen}
			Let $X$ be a projective variety defined over $\Qb$. Let $f\colon X\ra X$ be a surjective endomorphism. Then $f$ is int amplified if and only if  $f^*\colon N^1(X)_\RR\ra N^1(X)_\RR$ has all eigenvalues of modulus strictly larger then 1.
		\end{proposition}
		
		From \ref{thm:intampeigen} we see that the composition of int-amplified endomorphisms are int-amplified. We naturally obtain a sub-monoid of all surjective endomorphisms. 
		
		\begin{definition}[The monoid of surjective endomorphisms]\label{def:IntSur}
			Let $X$ be a projective variety defined over $\Qb$. We let $\Send(X)$ be the monoid of surjective endomorphisms of $X$. We let $\Iamp(X)$ be the collection of int-amplified endomorphisms of $X$.
		\end{definition}
		
		The idea behind int-amplified endomorphisms is that if $\Iamp(X)\neq \emptyset$ then we have a sub-monoid $\Iamp(X)\subseteq \Send(X)$ that can be used to study $\Send(X)$. We will see that the existence of $f\in \Iamp(X)$ has consequences for the birational geometry of $X$. The basic properties of int-amplified endomorphisms are the following.
		
		\begin{theorem}[Lemma 5.2 \cite{MR4070310}]\label{thm:propertiesofintam}
			Let $X,Y$ be normal projective varieties. Let $f\colon X\ra X$ and $g\colon Y\ra Y$ be surjective endomorphisms.
			
			\begin{enumerate}
				\item If $\pi\colon X\ra Y$ is a surjective endomorphism and $f$ is int-amplified with $g\circ\pi=\pi\circ f$ then $g$ is int-amplified.
				\item  If $\dim X=\dim Y$ and $\pi\colon X\dashrightarrow Y$ is a dominant rational map with $g\circ\pi=\pi\circ f$ then $g$ is int-amplified.
				\item If $X$ is $\QQ$-factorial and $f$ is int amplified then $-K_X\Num E$ where $E$ is an effective $\QQ$- divisor. In particular if $\Alb(X)=0$ then $\kappa(-K_X)\geq 0$. 
			\end{enumerate}
		\end{theorem}
Our interest in int-amplified endomorphisms comes from the following definition and application.
		\begin{definition}[The equivariant MMP:Meng-Zhang in \cite{MR4117085}]\label{def:equivariantMMP}
			
			Consider a sequence	of dominant rational maps
			
			\begin{align}\label{eq:equivraiantMMP}
				X_1\dashrightarrow X_2\dashrightarrow X_3\dots\dashrightarrow X_r
			\end{align}	
			
			such that each $X_i$ is a normal projective variety. Let $f=f_1\colon X_1\ra X_1$ be a surjective endomorphism. We say that \ref{eq:equivraiantMMP} is $f$-equivariant if there are surjective endomorphisms $f_i\colon X_i\ra X_i$ such that $g_i\circ f_{i+1}=f_i\circ g_i$ for all $i$, where $g_i\colon X_i\dashrightarrow X_{i+1}$ is the dominant rational mapping of \ref{eq:equivraiantMMP}.

		\end{definition}
		
		\begin{theorem}[The equivariant MMP of Meng-Zhang: Theorem \cite{MR4070310}]\label{thm:MengZhangMMP}
			Let $X$ be a $\QQ$-factorial projective variety	defined over $\Qb$ with at worst terminal singularities admitting an int-amplified endomorphism.
			
			\begin{enumerate}
				\item There are only finitely many $K_X$ negative extremal rays of $X$. Moreover if $f\colon X\ra X$ is a surjective endomorphism then there is some $n\in \ZZ_{>0}$ such that $f^n_*\colon \NE(X)_\RR\ra \NE(X)_\RR$ fixes every $K_X$ negative extremal ray. Let $R$ be any extremal $K_X$ negative extremal ray with contraction $\phi_R\colon X\ra Y_R$. Then there is a surjective endomorphism $g_R\colon Y_R\ra Y_R$ such that $g_R\circ \phi_R=\phi_R\circ f^n$. Moreover if $R$ is a flip and $\psi_R^+$ is the associated birational mapping $X\ra X_R^+$ then the induced rational mapping $f_R^+\colon X_R^+ \dashleftarrow X_R^+$ extends to a morphism $f_R^+\colon X_R^+\ra X_R^+$.
				
				\item Then for any surjective morphism $f\colon X\ra X$  there is some $n$ and a $f^n$ equivariant MMP for $f^n$ g given by \[X_1\dashrightarrow X_2\dashrightarrow X_3\dots\dashrightarrow X_r.\]
				Let $g_i\colon X_i\dashrightarrow X_{i+1}$. Then we have that.
				
				\begin{enumerate}
					\item Each $g_i$ is a contraction of a $K_X$ negative extremal ray.
					\item $X_r$ is a $Q$-Abelian variety. Note that $X_r$ might be a point. In fact there is there is a surjective endomorphism $h\colon A\ra X_r$ where $h$ is a finite and $A$ is an abelian variety. Moreover there is a surjective endomorphism $w\colon A\ra A$ such that $w\circ h=f_r\circ w$. The existence of $h$ is the definition of $Q$-Abelian, the theorem provides that the morphism $g_r$ commutes with a morphism of the covering abelian variety. In fact this holds for any surjective endomorphism of a $Q$-Abelian variety.
				\end{enumerate}	
				
				\item Let $f\colon X\ra X$ be a surjective endomorphism. Let \[X_1\dashrightarrow X_2\dashrightarrow X_3\dots\dashrightarrow X_r\] be any MMP where the $g_i\colon X_i\ra X_{i+1}$ are divisorial or fibering contractions. Then there is some $n$ such that there are surjective endomorphisms $f_i\colon X_{i}\ra X_i$ making the MMP $f$-equivariant. 
				\end{enumerate}
			\end{theorem}
			
\begin{subsection}{Canonical height functions and arithmetic degrees.}\label{subsec:canonicalheights}
To efficiently compute arithmetic degrees we will need the theory of heights. Our main reference is \cite{hindry2013diophantine}. Recall that height functions provide a way to turn geometric relationships involving divisors into arithmetic relationships involving height functions. For every line bundle $L$ on our normal projective variety $X$ there is an associated function $h_L\colon X(\overline{\QQ})\ra \RR$. The construction of $h_L$ is as follows. Write $L=D_1-D_2$ where $D_1,D_2$ are very ample divisors and then mimic construction in the introduction involving a height on projective space. Given a surjective endomorphism, $f\colon X\ra X$ we will study the dynamics of $f$ through the dynamics of its pull-back action on $N^1(X)_\RR$. The dynamics of a linear mapping is captured by its eigenvalues, so we are interested in those $D$ such that $f^*D\Lin \lambda D$ or $F^*D\Num \lambda D$ for some $\lambda\neq 0$. Associated to the divisor $D$ is a height function $h_D$. However, we have some extra data, namely that $D$ is preserved by $f$. This allows us to find a specific height function (rather then just an equivalence class of height functions) to study. Our main tool is the following.
\begin{theorem}[{\cite[Theorem 5]{MR4080251}},{\cite[Theorem 1.1]{MR1255693}}]\label{thm:canheight1}
		Let $X$ be a normal projective variety and $f\colon X\ra X$ a surjective endomorphism. Let $D\in \textnormal{CDiv}_\RR(X)$ and suppose that $f^*D\equiv_{\textnormal{Num}}\lambda D$ for some $\lambda>\sqrt{\lambda_1(f)}$. Then 
		\begin{enumerate}
			\item For all $P\in X(\bar{\QQ})$ the limit $\hat{h}_D(P)=\lim_{n\ra\infty }\dfrac{h_D(f^n(P))}{\lambda^n}$ exists for any choice of height function $h_D$ associated to $D$.
			\item We have that $\hat{h}_D(f^n(P))=\lambda^n\hat{h}_D(P)$ and that $\hat{h}_D=h_D+O(\sqrt{h^+_X}\ )$ where $h_X$ is any ample height on $X$ and $h_X^+=\max\{1,h_X\}$. If $f^*D\sim_\QQ \lambda D$ then $\hat{h}_D=h_D+O(1).$
			\item If $\hat{h}_D(P)\neq 0$ then $\alpha_f(P)\geq \lambda$.
			\item If $\lambda=\lambda_1(f)$ and $\hat{h}_D(P)\neq 0$ then $\alpha_f(P)=\lambda_1(f)$. 
			\item If we are working over a number field and $D$ is ample then $\hat{h}_D(P)=0\iff P$ is pre-periodic for $f$.   
		\end{enumerate}
	\end{theorem}
The following generalization by Kawaguchi and Silverman is what constrains the possible values of the arithmetic degree.
	\begin{theorem}[{\cite[Section 4]{MR4068299}}]\label{theorem:KS2} With the above notation in place let $P\in X(\bar{\QQ})$. 
		
		\begin{enumerate}
			\item Then $\alpha_f(P)=1$ or $\alpha_f(P)=\mid \lambda_i\mid$. More precisely suppose that $\hat{h}_{D_i}(P)\neq 0$ for some $1\leq i\leq \sigma$. Let $k$ be the smallest index with $\hat{h}_{D_k}(P)\neq 0$. Then $\alpha_f(P)=\mid \lambda_k\mid$. On the other hand if $\hat{h}_{D_k}(P)=0$ for all $1\leq k\leq \sigma$ then $\alpha_f(P)=1$. 
			\item In particular if $\mid \lambda_i\mid =\lambda_1(f)$ for $i=1,...,l$ and $\mid\lambda_{l+1}\mid <\lambda_1(f)$ then $\alpha_f(P)=\lambda_1(f)\iff \hat{h}_{D_i}(P)\neq 0$ for some $1\leq i\leq l$ . 
		\end{enumerate} 	
	\end{theorem}
The basic properties of the arithmetic degree are now given.
\begin{proposition}[Properties of the arithmetic degree: \cite{MR4080251} and \cite{MR4068299}]\label{prop:propofalpha}
		Let $X$ be a normal projective variety defined over $\Qb$. Let $f\colon X\ra X$ be a dominant morphism.
		
		\begin{enumerate}
		\item The limit $\lim_{n\ra\infty}h_H^+(f^n(P))^\frac{1}{n}$ exists.
			\item For all $P\in X(\Qb)$ we have that 
			\[\alpha_f(P)=\vert \lambda \vert \]
			for some eigenvalue $\lambda$ of $f^*$ acting on $N^1(X)_\RR$. We can also take $\lambda$ to be an eigenvalue of $f^*$ acting on $\Pic(X)$.  
			\item $\alpha_f(P)\leq \lambda_1(f)$.
		\end{enumerate}	
	\end{proposition}
We will also use the following.
	\begin{theorem}[{\cite[Theorem 1.8]{2007.15180}}]\label{theorem:MSS} Let $X$ be a projective variety defined over $\bar{\QQ}$ and $f\colon X\ra X$ be a surjective endomorphism with $\lambda_1(f)>1$. Then the set of points $P\in X(\bar{\QQ})$ with $\alpha_f(P)=\lambda_1(f)$ is Zariski dense.
	\end{theorem}		
	
\end{subsection}
When working with abelian varieties, we will need results about the endomorphism ring.

\begin{theorem}[\cite{MR1503260},
				page 201 \cite{MumfordAV}]\label{thm:AlbertClass}
				Let $A$ be a $g$-dimensional abelian variety over an algebraically closed field $k$ with a chosen polarization. Let $D=\Endo(A)_\QQ$ and $\prime$ the associated Rosati involution. Then $D$ is a simple $\QQ$-algebra, of finite dimension with center $K$ a number field. Set $K_0=\{x\in K: x^\prime=x\}$. Then the pair $(D,\prime)$ is one of the following four types. Set
			 \[e=[K:\QQ],e_0=[K_0:\QQ],m^2=[D:K]\]
				and let $\rho$ be the Picard number of $A$.
				\begin{enumerate}
					\item $K_0=K=D$ with $\prime=\textnormal{id}_D$ with $K$ a totally real number field. In this case $\rho=e$ and $e\mid g$.
					\item $[D:K]=4$ with $K=K_0$ with $K$ a totally real number field. $D$ is a quaternion algebra over $K$. Moreover for each real embedding $\sigma\colon K\hookrightarrow \RR$ we have $D\otimes_{K,
						\sigma} 
					\RR\cong M_2(\RR)$ and $D\otimes_\QQ \RR\cong \prod_{\sigma\colon K\hookrightarrow 
						\RR}M_2(\RR)$. The isomorphism can be chosen so that the involution can be taken to be $(M_1,...,M_e)^\prime=(M_1^t,...,M_e^t)$. 
					
					We have $\rho=3e$ and $2e\mid g$. 
					\item $[D,K]=4$ with $K=K_0$ with $K$ a totally real number field. $D$ is a quaternion algebra over $K$. Moreover for each real embedding $\sigma\colon K\hookrightarrow \RR$ we have $D\otimes_{K,
						\sigma} 
					\RR\cong \mathbb{H}$ and $D\otimes_\QQ \RR\cong \prod_{\sigma\colon K\hookrightarrow 
						\RR}\mathbb{H}$. The isomorphism can be chosen so that the involution can be taken to be $(a_1,...,a_e)^\prime=(\bar{a}_1,...,\bar{a}_e)$. 
					
					We have $\rho=e$ and $2e\mid g$.
					\item$[D:K]=m^2$ with $K_0$ is a totally real number field, and $K$ is a totally imaginary quadratic extension of $K_0$. \[D\otimes_\QQ \RR\cong \prod_{\sigma\colon K_0\hookrightarrow 
						\RR}M_m(\CC)\] The isomorphism can be chosen so that the involution can be taken to be $(M_1,...,M_{e_0})^\prime=(\bar{M}^t_1,...,\bar{M}^t_{e_0})$. We have $\rho=e_0m^2$ and $e_0m^2\mid g$.
				\end{enumerate}

			\end{theorem}	
		When working with Abelian varieties, we will use the following. 	
	\begin{theorem}\label{thm:absetup}
				Let $A$ be an abelian variety defined over $\bar{\QQ}$. Fix an ample divisor $H$ on $A$. Then
				
				\begin{enumerate}
					\item The image of $N^1(A)_\QQ$ is precisely the set of $\alpha\in \Endo(A)_\QQ$ with $\alpha^\prime=\alpha$. That is the Neron-Severi space may be identified with the space of endomorphisms fixed by the Rosati involution.
					\item (\cite[Lemma 24]{MR4068299}) Let $f\colon A\ra A$ be an isogeny and $D\in N^1(A)_\RR$. Then 
					\[\Phi_{f^*D}=f^\prime\circ \Phi_D\circ f.\]
					\end{enumerate}

			\end{theorem}

We will also need the following. Its interest is primarily to allow us to determine when $f^*\colon N^1(X)\ra N^1(X)$ is given by a scalar multiple of the identity matrix.

\begin{theorem}[{\cite[Theorem 4.8]{MR1832167}}]\label{thm:conethm}
			Let $A$ be a real $n\times n$ matrix. Then the following are equivalent.
			
			\begin{enumerate}
				\item $A$ is non-zero, diagonalizable, and all eigenvalues of $A$ have the same modulus, with $\rho(A)$ being an eigenvalue.
				\item There is a proper cone $K$ in $\RR^n$ with $A(K)\subseteq K$ and $A$ has an eigen-vector in the interior of $K$.
				
			\end{enumerate}
		\end{theorem}

\end{section}

\section{Arithmetic eigenvalues}\label{sec:Realizabilitymain}
We begin our study of arithmetic degrees. Note that not every eigenvalue will be an arithmetic degree. If $f$ is an automorphism with eigenvalue $\lambda>1$ then $f^*$ may have an eigenvalue $\frac{1}{\lambda}<1$. Since $\alpha_f(P)\geq 1$ for all $P$ we see that such eigenvalues are not arithmetic in this sense. On the other hand, recall that in \cite{MR3871505} it was proven that $\lambda_1(f)$ is arithmetic. Our question is therefore most meaningful for eigenvalues $\mu$ with $\vert \mu\vert <\lambda_1(f)$. Recall that the sAND conjecture (conjecture \ref{conj:sAND}) predicts that the set of points with $\alpha_f(P)<\lambda_1(f)$ is not Zariski dense. The following result is elementary for our work. 
\begin{proposition}\label{prop:dialprop}
	Let $X$ be a normal projective variety defined over $\overline{\QQ}$. Suppose that $X$ admits a surjective endomorphism $f\colon X\ra X$. If $f^*\colon N^1(X)_\RR\ra N^1(X)_\RR$ acts by multiplication by a scalar, then $f$ has arithmetic eigenvalues. 
\end{proposition}
\begin{proof}
	Suppose that the action of $f^*$ on $N^1(X)_\RR$ is given by multiplication by $\lambda$. Since $f^*$ is defined over $\mathbb{Z}$ we have that $\lambda\in \ZZ$. If $\vert \lambda\vert =1$ then $f$ has no potential arithmetic degrees as $\alpha_f(P)=1$ for all $P\in X(\overline{\QQ})$ by part 4 of \ref{prop:propofalpha}. Now assume that $\vert \lambda\vert >1$. By \ref{theorem:MSS} there is a point $P$ with $\alpha_f(P)=\vert \lambda\vert$. Therefore, all potential arithmetic degrees are realized. 
\end{proof}
Dynamics studies the behavior of maps under iteration; it is valuable to know if a property of maps is preserved by iteration.
\begin{proposition}\label{prop:iterationprop}
	Let $X$ be a normal projective variety defined over $\overline{\QQ}$ and let $f\colon X\ra X$ be a surjective endomorphism.
	
	\begin{enumerate}
		\item If $\gamma$ is a potential arithmetic degree of $f^m$ then $\gamma = \mu^m$ where $\mu$ a potential arithmetic degree of $f$.
		\item $\gamma$ is realizable as an arithmetic degree $\iff$  $\mu$ is realizable as an arithmetic degree.
	\end{enumerate}   
	If $f^m$ has arithmetic eigenvalues then $f$ has arithmetic eigenvalues.  
\end{proposition}
\begin{proof}
	The eigenvalues of $(f^m)^*$ are $m^{th}$ powers of the eigenvalues of $f^*$. So $\gamma=\mu^m$ for some eigenvalue of $f^*$. Now let $\alpha_{f^m}(P)= \vert \gamma\vert=\vert \mu\vert^m$. Then we have $\alpha_f(P)^m=\alpha_{f^m}(P)=\vert \mu\vert ^m$ and taking $m^{th}$ roots gives the desired result. Now suppose that every potential arithmetic degree of $f^m$ is an arithmetic degree. Let $\mu$ be a potential arithmetic degree of $f$. Then $\mu^m$ is a potential arithmetic degree of $f^m$ and the result follows from the above calculation. 	
\end{proof}
The result \ref{prop:iterationprop} is pleasing because it shows that question \ref{ques:ques2} is equivalent for all iterations of $f$.
\begin{section}{Realizability for abelian varieties.}\label{sec:realizabilityforabelianvarieties}
	In this section we consider question \ref{ques:ques2} when $f\colon A\ra A$ is an isogeny of an abelian variety defined over $\overline{\QQ}$. We give a negative answer to question \ref{ques:ques2} and give an algebraic interpretation of dynamical questions in terms of the endomorphism algebra of $A$. To understand $f^*$ acting on $N^1(A)_\RR$ it suffices to understand the eigenvalues of the twisted conjugation action of $f$ on the fixed points of the Rosati involution. Here we review the relevant facts about Neron-Severi groups of Abelian varieties.  \begin{align}\label{eq:symstuff}
		&N^1(A)_\RR=\Endo(A)_
		\RR^{\textnormal{Sym}}=\{\alpha\in \Endo(A)_\RR:\alpha^\prime=\alpha\}.\\
		&\theta_f\colon \Endo(A)_
		\RR^{\textnormal{Sym}}\ra \Endo(A)_
		\RR^{\textnormal{Sym}},\alpha\mapsto f^\prime \circ \alpha\circ f\textnormal{ describes the action of }f^*\colon N^1(A)_\RR\ra N^1(A)_\RR. 
	\end{align} 
	The eigenvalues of $f^*$ acting on $N^1(X)_\RR$ are thus interpreted as those $\alpha\in \Endo(A)_
	\RR^{\textnormal{Sym}}$ such that $\theta_f \alpha=\lambda\alpha$. The benefit of interpreting the dynamics of $f^*$ in this light is that we have a good understanding of the possible endomorphism groups of abelian varieties, at least after tensoring with $\QQ$. This will allow us to get a refined notion of the possible dynamics and to prove that the existence of endomorphisms of abelian varieties with prescribed properties. Here we will use \ref{thm:AlbertClass}. We will need the following results of Kawaguchi and Silverman.
\begin{theorem}[{\cite[Theorem 29]{MR4068299}}]\label{thm:silv1}
		Let $A/\overline{\QQ}$ be an abelian variety. Let $D\in\textnormal{Div}_\RR(A)$ be a real nef divisor class with $\hat{q}_{A,D}$ the quadratic part of the canonical height associated to $D$. 
		\begin{enumerate}
			\item There is a unique proper abelian subvariety $B_D\subseteq A$ such that
			\[\{x\in A(\overline{\QQ}):\hat{q}_{A,D}(x)=0\}=B_D(\overline{\QQ})+A(\overline{\QQ})_{\textnormal{tors}}.\]
			\item Suppose further that $f\colon A\ra A$ is an isogeny defined over $\overline{\QQ}$ with $f^*D=\lambda_1(f)D$ in $N^1(A)_\RR$. Then $\hat{q}_{A,D}(P)\geq 0$ for all $P\in A(\overline{\QQ})$ and $\hat{q}_{A,D}(P)>0\Rightarrow \alpha_f(P)=\lambda_1(f)$.
		\end{enumerate}
	\end{theorem}
	We now have the following easy consequence of our set up. Recall that the Picard number of an abelian variety is related to its \emph{type}. See theorem \ref{thm:AlbertClass} for the details. 
	\begin{proposition}
		Let $A$ be a geometrically simple abelian variety of type I or III. Then any isogeny $f$ acts on $N^1(A)_\RR$ by scalar multiplication. Furthermore, if $\lambda_1(f)>1$ then any such action is polarized and $f$ has arithmetic eigenvalues. 
	\end{proposition}
	\begin{proof}
		In type I we have that the endomorphism ring of $A$ is commutative and so the action of $f^*$ on $N^1(A)_\RR$ is given by
		\[\alpha\mapsto f^\prime \circ \alpha \circ f=(f^\prime\circ f) \alpha\]
		by Theorem \ref{thm:absetup}. As $f^\prime\circ f$ is a real number, $f^*$ acts by dilation on $N^1(X)_\RR$. In type III we have that $D$ is a quaternion algebra and the Rosati involution is the standard involution. Thus the points $\alpha$ with $\alpha^\prime=\alpha$ is precisely the center $K$. The same argument then tells us that $f^*$ acts by a scalar multiplication on the Neron-Severi space. Now suppose that $\lambda_1(f)>1$. Then if $f^*$ acts by multiplication by a scalar $\mu$ we have that $\vert \mu\vert=\lambda_1(f)>1$. Thus $f^*$ is polarized. By Proposition \ref{prop:dialprop} all potential arithmetic degrees are realized.  
	\end{proof} 	
	In conclusion, the dynamics of a surjective endomorphism of a simple abelian variety of type I or III is easy to understand because there is an ample canonical height function. The other remaining cases are more complicated. The additional cases give rise to isogenies with eigenvalues which are not realizable as arithmetic degrees.  
	\begin{theorem}\label{thm:MainCounterexample}
		For each $g\in 2\ZZ_{>0}$ there is a simple abelian variety $A$  of dimension $g$ defined over a number field $K$ with $\rho(A)=3$ equipped with a surjective endomorphism $f\colon A\ra A$ that has the following properties.
		\begin{enumerate}
			\item $f^*\colon N^1(A)_\QQ\rightarrow N^1(A)_\QQ$ has eigenvalues $a^2>ab>b^2>0$ for some $a,b\in \mathbb{Z}$.
			\item $\alpha_f(P)=a^2$ for all $P\notin A(\overline{K})_{\textnormal{tors}}$. In particular, $\alpha_f(P)\in \{1,a^2\}$.
			\item The eigendivisors of $a^2,b^2$ are nef while the eigendivisor of $ab$ is not. 
		\end{enumerate}
	\end{theorem}	
	\begin{proof}
		Let $A$ be a simple abelian variety with $\Endo(A)_\QQ=M_2(\QQ)$ with involution given by the transpose. Such abelian varieties exists by general results of (Oort/Shimura). See for example \cite{MR977774} and \cite{MR949273}. Consider an endomorphism of the form $f=\begin{bmatrix} a & 0\\ 0 &b \end{bmatrix}$ with $a,b\in \ZZ$ and $a>b$.  In this case the Neron-Severi group is of rank 3. One checks directly that the eigenvectors of the twisted conjugation action are given by
		\[\begin{bmatrix} 1 & 0\\ 0 &0 \end{bmatrix},\begin{bmatrix} 0 & 0\\ 0 &1 \end{bmatrix},\begin{bmatrix} 0 & 1\\ 1 &0 \end{bmatrix}\]
		with eigenvalues $a^2,b^2,ab$ respectively. As $A$ is simple, the set of points where $\alpha_f(P)<\lambda_1(f)=a^2$ is given by the torsion points by Theorem \ref{thm:silv1}, which all have arithmetic degree $1$. Thus $\alpha_f(P)=a^2$ or $1$ and never can $b^2$ or $ab$ when $b\neq 1$. It is known (see the proof of \cite[Proposition 26]{MR4068299} ) that a symmetric matrix representing a divisor is ample if and only if it has positive eigenvalues, and it is nef if and only if it has non-negative eigenvalues. We see that $a^2,b^2$ have nef eigendivisors but $ab$ does not as $-1$ is an eigenvalue of $\begin{bmatrix} 0 & 1\\ 1 &0 \end{bmatrix}$.  
	\end{proof}
	We can give some geometric insight into this situation by considering Abelian surfaces.
	\begin{example}
		\normalfont Let $A$ be a simple abelian surface, Let $f\colon A\ra A$ be a surjective endomorphism. Suppose that $\mu$ is an eigenvalue of $f^*$ acting on $N^1(A)_\RR$ space with $1<\vert \mu \vert <\lambda_1(f)$. If there is a point $P$ with $\alpha_f(P)=\vert \mu\vert$ we must have that $\overline{\OO_f(P)}=V$ is a curve on $A$. This is because if $V=A$ then by the Kawaguchi-Silverman conjecture for abelian varieties we have $\alpha_f(P)=\lambda_1(f)$. On the other hand if $V$ is zero dimensional then $\alpha_f(P)=1$. So $V$ must be a curve on $A$. After iterating $f$ we may assume that $V$ is irreducible and we obtain by restriction $f\colon V\ra V$. Thus $V$ is an irreducible curve on $A$ which means that the genus of its normalization is at least $2$. Since we have $f\colon V\ra V$ we obtain a surjective morphism $\tilde{f}$ on the normalization of $V$. Since $f\mid_V$  has a dense orbit by construction, $\tilde{f}$ has a dense orbit. However, this is impossible, since a genus $g>2$ smooth curve has no surjective endomorphism that is not an automorphism, and the automorphism group of such a curve is finite. We see that $V$ is either zero dimensional, or all of $A$. In other words, every point is either pre-periodic or has a dense orbit. Thus to construct an endomorphism with non-realizable eigenvalues, it suffices to find an endomorphism of a simple abelian surface that admits a surjective endomorphism with two integral eigenvalues of different size. 
	\end{example}	
	We now turn to non-simple abelian varieties.
	\begin{theorem}
		Let $A=A_1\times A_2\times \dots \times A_n$ where the $A_i$ are simple pairwise non-isogenous abelian varieties such that if $g_i\colon A_i\ra A_i$ is an isogeny then all potential arithmetic degrees of $g_i$ are realizable as an arithmetic degree. If $f\colon A\ra A$ is an isogeny then every potential arithmetic degree of $f$ is realizable as an arithmetic degree.
	\end{theorem}
	\begin{proof}
		Let $f\colon A\ra A$ be an isogeny. Write $f=f_1+f_2+\dots + f_n$ where $f_i\colon A\ra A_i$. Fix an integer $j\neq i$.  Let $p_j\colon A_j\hookrightarrow A$ be the morphism that sends $a\mapsto (a_i)$ where $a_i=O_{A_i}$ for $i\neq j$ and $a_j=a$. So for example $p_1\colon A_1\ra A $ is the canonical map $a\mapsto (a,O_{A_2},\dots,0_{A_n})$. Then $h_{ji}=f_i\circ p_j\colon A_j\ra A_i$ is a homomorphism of simple abelian varieties. Since $A_j$ is simple the kernel is either finite or all of $A_j$. Towards a contradiction suppose that the kernel was finite. Then the image is either all of $A_i$ or is $O_{A_i}$. If the image was all of $A_i$ then $h_{ji}$ is a surjective homomorphism with finite kernel, meaning it is an isogeny. This is a contradiction as $A_i$ and $A_j$ are not isogenous. So the image is $O_{A_i}$. This contradicts that the kernel is finite. So $h_{ji}$ is the constant mapping to $O_{A_i}$. Since we can write $f_i(a_1,...,a_n)=h_{1i}(a_1)+\dots+h_{ni}(a_n)$ we have that $f_i(a_1,...,a_n)=h_{ii}(a_i)$. Consequently we have that $f_i\colon A\ra A_i$ is induced by some isogeny $h_i\colon A_i\ra A_i$. Thus we have that $f=h_1+\dots+h_n$. In this case we have that since the $A_i$ are pairwise all non-isogenous and simple that
		\[\rho(A)=\sum_{i=1}^n \rho(A_i)\]
		by \cite[2.3]{AbelianPicardHulek}. Therefore, we have that $N^1(A)=\prod_{i=1}^n \pi_i^*N^1(A_i)$ where $\pi_i\colon A\ra A_i$ is the projection. Since $f=h_1+\dots +h_n$ we have that the eigenvalues of $f^*$ are all of the form $h_i^*\pi_i^*H=\mu H$ where $H$ is some class on $A_i$. Thus the potential arithmetic degrees of $f$ are the potential arithmetic degrees of the $h_i$. Suppose that $\mu_i$ is a potential arithmetic degree of $h_i$. Then there is a point $Q_i$ in $A_i$ such that $\alpha_{h_i}(Q_i)=\vert \mu_i\vert$ by assumption. Then set $P=(P_1,\dots P_n)$ with $P_j=O_{A_i}$ if $j\neq i$ and $P_i=Q_i$. We have that
		\[\alpha_f(P)=\max_{w=1}^n\{\alpha_{h_w}(P_w)\}=\alpha_{h_i}(Q_i)=\vert \mu_i\vert.\] As every potential arithmetic degree of $f$ arises in this manner we have the desired result. 
		\end{proof}
On the other hand, if we allow powers of a simple abelian variety then potential arithmetic degrees may be not be realizable. 
		\begin{example}\label{example:ExE}
		\normalfont Let $A=E\times E$ where $E$ does not have complex multiplication. Then $\rho(A)=3$ by \cite[2.6]{AbelianPicardHulek} and $\textnormal{End}(A)_\QQ=M_2(\QQ)$.  Consider the isogeny $f(P,Q)=(aP,bQ)$ where $a,b>0$ are integers. Note that $f$ corresponds to the matrix $\begin{bmatrix} a& 0\\ 0 & b \end{bmatrix}$ as in Theorem \ref{thm:MainCounterexample}. It is clear that $f^*$ has eigenvectors $a^2,b^2$ as $[n]\colon E\ra E$ acts on $N^1(E)_\RR$ by multiplication by $n^2$. It remains to find the third eigenvalue. We have that $N^1(E\times E)_\RR$ can be identified with symmetric matrices and $f^*$ acts by 
		\[A\mapsto \begin{bmatrix} a& 0\\ 0 & b \end{bmatrix}^t A \begin{bmatrix} a& 0\\ 0 & b \end{bmatrix}. \]
The final eigenvalue of $f^*$ is then represented by
\[\begin{bmatrix} 0 &1\\ 1& 0 \end{bmatrix}\]
as
\[ \begin{bmatrix} a& 0\\ 0 & b \end{bmatrix}^t \begin{bmatrix} 0 &1\\ 1& 0 \end{bmatrix} \begin{bmatrix} a& 0\\ 0 & b \end{bmatrix}=ab\begin{bmatrix} 0 &1\\ 1& 0 \end{bmatrix}. \]	We see that that $f^*\colon N^1(A)_\RR\ra N^1(A)_\RR$ has 3 eigenvalues, $a^2,ab,b^2$. Note that $ab$ does not have a nef eigendivisor, as $\begin{bmatrix} 0 &1\\ 1& 0 \end{bmatrix}$ has a negative eigenvalue. Let $\hat{h}_E$ be the canonical height on $E$. Then we have that
\[\hat{h}_A(P,Q)=\hat{h}_E(P)+\hat{h}_E(Q)\] 
is an ample height on $A$. Then  
\[\hat{h}_A(f^n(P,Q))=\hat{h}_A(a^nP,b^nQ)=a^{2n}\hat{h}_E(P)+b^{2n}\hat{h}_E(Q).\]
Thus we see that
\[\lim_{n\ra \infty}\hat{h}_A^+(f^n(P,Q))^{1/n}\in \{a^2,b^2,1\}.\] 
In fact all of the above values are realizable, by taking $(P,O),(O,P)$ and $(O,O)$ where $O$ is the identity element of $E$ and $P$ is a point with $\hat{h}_E(P)\neq 0$. First take $a>b>1$. In this case we see that $a^2$ has eigenvector $\begin{bmatrix}1 & 0 \\ 0 & 0 \end{bmatrix}$ and $b^2$ has eigenvector $\begin{bmatrix} 0 & 0 \\ 0 & 1 \end{bmatrix}$ which are both nef as these matrices have non-negative eigenvalues. On the other hand the non-realizable eigenvalue $ab$ has a non-nef eigenvector  $\begin{bmatrix} 0 & 1 \\ 1 & 0 \end{bmatrix}$. On the other hand if we take $a>1,b=1$ then we obtain examples of a surjective endomorphism with eigenvalues $a^2,a,1$ but $a$ is not realizable.
	\end{example}
The above example is illustrative in the following sense. Let $X,Y$ be normal projective $\QQ$-factorial varieties. Suppose that $f\colon X\ra X$ and $g\colon Y\ra Y$ are surjective endomorphisms. If $\rho(X\times Y)>\rho(X)\times \rho(Y)$ then $(f\times g)^*$ will have mixed eigendivisors that do not arise as the pull back of some divisor on $X$ or $Y$. On the other hand,
	\[\alpha_{f\times g}(P,Q)=\max\{\alpha_f(P),\alpha_g(Q)\}.\]
	If these mixed eigendivisors have mixed eigenvalues that do not appear as eigenvalues on $X$ and $Y$ then they will not be realizable. In terms of abelian varieties, the Picard number of $A^k$ is always strictly larger then the Picard number of $A$ when $k>1$ and $A$ is simple by \cite[2.4]{AbelianPicardHulek}. Thus we might always expect that $A^k$ always has an isogeny with unrealizable potential arithmetic degree. Indeed, this is the case for squares of elliptic curves without CM by \ref{example:ExE}.
	\begin{corollary}\label{cor:Kodiradimrealizability}
	For any integer $d>1$ there is a smooth projective variety $X$ with $\dim X=d$ such that there is a surjective endomorphism $f\colon X\ra X$ with $\lambda_1(f)>1$ and $f$ has does not have arithmetic eigenvalues. If $d\geq 3$ and $\kappa\in \{-\infty,0,1,\dots , d-2\}$ then $X$ may be chosen with $\kappa(X)=\kappa$. If $\kappa\in \{-\infty,0\}$ then the surjective endomorphism may be taken to be int-amplified.
	\end{corollary}
\begin{proof}
Let $E$ be a fixed elliptic curve without CM. If $d=2$ then take $X=E\times E$. By example \ref{example:ExE} there is an isogeny with positive dynamical degree but with non-arithmetic eigenvalues. Now let $d\geq 3$ and $\kappa\in \{-\infty,0,1,\dots , d-2\}$. If $\kappa=-\infty$ then take $X=(\PP^1)^{d-2}\times E^2$. Let $f$ be as in \ref{example:ExE} with eigenvalues $a^2>ab>b^2>1$ and $ab$ not arithmetic. Let $h\colon (\PP^1)^{d-2}\ra (\PP^1)^{d-2}$ be any surjective endomorphism of the form $h_1\times \dots\times h_{d-2}$ with $\lambda_1(h_i)\neq ab$ for all $1\leq i\leq d-2$. Then given any point $P=(Q,R)\in X$ where $Q=(Q_1,\dots,Q_{d-2})$ we have that
\[\alpha_{h\times f}(P)=\max\{\alpha_h(Q),\alpha_f(R)\}.\]
Since \[\alpha_h(Q)=\max_{i=1}^{d-2}\{\alpha_{h_i}(Q_i)\}\in \{1,\lambda_1(h_1),\dots,\lambda_1(h_{d-2})\}\]
and
\[\alpha_f(R)\in \{1,a^2,b^2\}\]
we have that
\[\alpha_{h\times f}(P)\in \{1,\lambda_1(h_1),\dots,\lambda_1(h_{d-2}),b^2,a^2\}.\]
By construction $ab\notin \{1,\lambda_1(h_1),\dots,\lambda_1(h_{d-2}),b^2,a^2\}$ and so $ab$ is not an arithmetic degree of $h\times f$. On the other hand $ab$ is an eigenvalue of $(h\times f)^*$ as it is an eigenvalue of $f^*$. If we choose $\lambda_1(h_i)>1$ for all $i$ then the endomorphism is int-amplified as all the eigenvalues are strictly larger then $1$. Now let $\kappa\geq 0$. If $\kappa=0$ take $X=E^d$. Let $f$ be as above and define $g=\colon E^d\ra E^d$ as $g=(g_1,...,g_{d-2},f)\colon E^{d-2}\times E^2\ra E^{d-2}\times E^2$ where the $g_i$ are surjective endomorphisms with $\lambda_1(g_i)\neq ab$. The argument given above gives that $ab$ is an eigenvalue of $g^*$ but not arithmetic. If we take $\lambda_1(g_i)>1$ for all $i$ we once more obtain that $g$ is int-amplified. Finally take $\kappa>0$. Set $X=C^\kappa\times E^{d-\kappa}$ where $C$ is any smooth curve of genus at least 2. Since $\kappa\leq d-2$ we may write $X=C^\kappa\times E^{d-\kappa-2}\times E^2$. Now let $g\colon X\ra X$ be the morphism which is the identity on the first $d-2$ factors and $f$ on $E^2$. That is 
\[g=\textnormal{identity}\times f\colon (C^{\kappa}\times E^{d-\kappa-2})\times E^2\ra (C^{\kappa}\times E^{d-\kappa-2})\times E^2.\]
Let $\pi\colon (C^{\kappa}\times E^{d-\kappa-2})\times E^2\ra E^2$ be the projection. Then given any point $P\in X(\Qb)$ we have
\[\alpha_g(P)=\alpha_f(\pi(P))\in \{a^2,b^2,1\}.\]
As above we obtain that $ab$ is a non-arithmetic eigenvalue of $g^*$ as needed. 
\end{proof}
\begin{remark}
	\normalfont In \ref{cor:Kodiradimrealizability} the examples of potential arithmetic degrees which are not realizable occur when the potential arithmetic degree does not have a nef eigendivisor. This observation leads to the following question.
	\end{remark}
	\begin{question}\label{ques:realizablev2}
		Let $X$ be a normal projective variety defined over $\overline{\QQ}$ and let $f\colon X\ra X$ be a surjective endomorphism. Let $\mu$ be a potential arithmetic eigenvalue of $f$ with a nef eigendivisor in $N^1(X)_\RR$. Then is $\mu$ arithmetic?   
	\end{question}
In light of the above discussion, one might wonder if eigenvalues of $f^*$ that are constructed in some geometrically meaningful sense are realizable. Indeed, $\lambda_1(f)$ has a geometric realization and is arithmetic. From the point of view of question \ref{ques:realizablev2} it also has a nef eigendivisor. 
	\begin{question}
		Let $X,Y$ be smooth projective varieties over $\overline{\QQ}$. Consider a diagram
		\[\xymatrix{X\ar[r]^f\ar[d]_\pi & X\ar[d]^\pi\\ Y\ar[r]_g & Y}\]
		with $f,g,\pi$ being surjective. Then is $\lambda_1(f\mid_\pi)$ realizable as an arithmetic degree?
	\end{question}
\end{section}
\begin{section}{Realizability when $\Alb(X)=0$ and $\kappa(X)=-\infty$.}\label{sec:Real-infty}
In light of $\ref{cor:Kodiradimrealizability}$ and its proof, to construct varieties with endomorphisms that possess arithmetic eigenvalues we must eliminate the possibility of morphisms to an abelian variety that has non-arithmetic eigenvalues. To ensure this we will demand that $\Alb(X)=0$. We first consider the case of smooth surfaces, and then of small Picard numbers. We then piece together some small results to be used in the future.
\begin{proposition}\label{prop:TSurfaceTriv}
	Let $X$ be a smooth toric surface and let $f\colon X\ra X$ be a surjective endomorphism that is not an automorphism with $\lambda_1(f)>1$. Then $f^s$ is polarized for some $s$ or $X=\PP^1\times \PP^1$
\end{proposition}
\begin{proof}
	There is a classification of smooth toric surfaces; $X$ is either $\PP^2,\McH_r$ for $r\geq 2$ (The Hirzebruch surface) or $\PP^1\times\PP^1$ or a series of blow ups of torus invariant points of one of these varieties. See \cite[10.4]{CLOToric} for the details. Since the nef cone of any toric variety is finitely generated, after iterating $f$ we may assume that $f^*$ fixes the rays of the nef cone. First suppose that $X$ is a series of blow ups at torus invariant points starting at $\PP^2$. Then we have a diagram
	\[\xymatrix{X=X_0\ar[r]^{g_0} & X_1\ar[r]^{g_1}&\dots\ar[r]^{g_{r-1}}& X_r=\PP^2}\]
	where each $g_i$ is a divisorial extremal contraction. By \ref{thm:MengZhangMMP} there is some $n\geq 1$ such that there are surjective endomorphisms $f_i\colon X_i\ra X_i$ making the above sequence $f^n$-equivariant. Then $f_{r}\colon \PP^2\ra \PP^2$ has an ample eigendivisor $H_{r}=\OO_{\PP^2}(1)$. Since $g_{r-1}\colon X_{r-1}\ra \PP^2$ is a blow up we have $H_{r-1}=g_{r-1}^*H_r$ is a nef and big divisor that is not ample. Furthermore since $g_{r-1}\circ f_{r-1}=f_r\circ g_{r-1}$ we have that $f_{r-1}^*g_{r-1}^*H_r=g_{r-1}^*f_r^*H_r$ and so $f_{r-1}$ has a nef and big eigendivisor. By induction we have that $f_0$ has a nef and big eigendivisor. So by \ref{thm:conethm} applied to the big cone and its closure we have that  $f$ is polarized as $\lambda_1(f)>1$. One can check that $\PP^1\times\PP^1$ blown up at a point is isomorphic to $\PP^2$ blown up at two points. Therefore the previous argument applies. If $X=\McH_r$ then one may directly check that $\McH_r$ has a nef and big divisor that is not ample. In other words $\McH_r$ has effective divisors that are not nef. Since $\McH_r$ has Picard number two it must be an eigendivisor.  The argument just given gives for $\PP^2$ now applies to show that series of blow ups starting at $\McH_r$ has a nef and big eigendivisor and so is polarized.
\end{proof}

The above result is sharp in the sense that if $f$ is a degree $d>1$ morphism $f\colon \PP^1\ra \PP^1$ and $g\colon \PP^1\ra\PP^1$ is a degree $d^\prime>1$ morphism with $d\neq d^\prime$ then $f\times g\colon \PP^1\times\PP^1\ra\PP^1\times\PP^1$ is a surjective morphism with $\lambda_1(f\times g)>1$ but $f\times g$ is not polarized. 

\begin{proposition}
	Let $X$ be a smooth projective surface defined over number field with $\Alb(X)=0$ and $\kappa(X)\leq 0$. Let $f\colon X\ra X$ be a surjective endomorphism. Then $f$ has arithmetic eigenvalues. 
\end{proposition}
\begin{proof}
	If $\lambda_1(f)=1$ then there is nothing to prove, so we may assume that $\lambda_1(f)>1$. First suppose that $f\colon X\ra X$ is an automorphism.  Then by \cite[2.4.3]{MR3289919} we have that $\lambda_1(f)$ is the only potential arithmetic degree and by Theorem \ref{theorem:MSS} it is realizable. So we may assume that $f$ is not an automorphism. Then by \cite[Page 1]{MR2154100} we have that $X$ is a smooth toric surface. If $X$ is not $\PP^1\times \PP^1$ then $f$ is polarized by Proposition \ref{prop:TSurfaceTriv} and every arithmetic degree is realizable. We may thus assume that $X=\PP^1\times\PP^1$. After iterating $f$ we may by \ref{thm:MengZhangMMP} applied to the two fibering contractions of $\PP^1\times \PP^1$ assume that $f=g_1\times g_2\colon \PP^1\times\PP^1\ra\PP^1\times\PP^1$ with $g_i\colon \PP^1\ra \PP^1$ a degree $d_i$ morphism. Without loss of generality assume that $d_1>1$. Choose a point $P\in \PP^1$ such that $\alpha_{g_1}(P)=d_1$ and $P_2\in \PP^1$ a pre-periodic point for $g_2$. Then $\alpha_{g_1\times g_2}((P_1,P_2))=d_1$ as needed. 
\end{proof}
The last remaining case when $\Alb(X)=0$ according to \cite{MR2154100} is when $\kappa(X)=1$ which we do not treat as to not go too far afield into the world of elliptic surfaces. After developing some theory we will return to the case of certain  singular surfaces. 

\begin{proposition}\label{prop:smallpicardnumber}
	Let $X$ be a projective normal $\QQ$-factorial variety with at worst terminal singularities. Suppose that we have the Picard number $\rho(X)$ is $2$. If $X$ admits an extremal contraction that is not of fibering type then $(f^2)^*$ acts by scalar multiplication. In particular, if $X$ is not a Mori fiber space and not minimal then given any surjective endomorphism $f$ of $X$ we have that $(f^2)^*$ $X$ acts on $N^1(X)_\RR$ by scalar multiplication.
\end{proposition}
\begin{proof}
	By hypothesis there is an extremal ray $R$ of $X$ an a contraction $\phi\colon X\ra Y$ of birational type. Let $H$ be an ample divisor on $Y$. Then $\phi^* Y$ is a big and nef divisor on $X$ that is not ample. In particular, it must lie on the boundary of $\Nef(X)$. So $\Nef(X)$ has a ray which is big. Then $(f^*)^2$ has a nef and big eigenvector. By \ref{thm:conethm} we have that $(f^*)^2$ acts by scalar multiplication.
	\end{proof}
We now describe a method of producing arithmetic degrees.	
\begin{proposition}\label{prop:fiberingbackbone}
	Let $X$ be a projective normal $\QQ$-factorial variety with at worst terminal singularities and $\Alb(X)=0$. Suppose we have 
	\[\xymatrix{X\ar[r]^f\ar[d]_\phi & X\ar[d]^\phi\\ Y\ar[r]_g& Y}\]
where $\phi$ is a fibering extremal contraction and $f,g$ surjective endomorphisms. Let $\mu_1,...,\mu_{\rho-1}$ be the eigenvalues of $g^*$ acting on $N^1(Y)_\RR$. Suppose that $\lambda$ is an eigenvalue of $f^*$ that is not an eigenvalue of $g^*$ and that $D$ is a non-zero nef divisor class with $f^*D\Lin \lambda D$. In addition suppose that $\vert\lambda\vert >1$ and $D$ does not lie in the image of $\phi^*(N^1(Y)_\RR)$. If there is a point $Q\in Y(\overline{\QQ})$ with $g(Q)=Q$ then there is a point $P\in \pi^{-1}(Q)$ with $\alpha_f(P)=\lambda$.
\end{proposition}
\begin{proof}
	Since $\phi$ is fibering we have $\rho(X)-1=\rho(Y)$ where $\rho(X)$ is the Picard number of $X$. We may choose an ample divisor $H$ on $Y$ with $A=D+\phi^*H$ being ample on $X$. Otherwise, there would be a full dimensional subset of the boundary of $\Nef(X)$ which is impossible. Set $F=\pi^{-1}(Q)$. Let $\hat{h}_{D}$ be the canonical height function of $D$ (see \ref{thm:canheight1}) defined as
	\[\hat{h}_D(x)=\lim_{n\ra \infty}\frac{1}{\lambda^n}h_D(f^n(x)).\] 
	Notice that
	\[D+\pi^*H\mid_{F}=D\mid_F\]
	since the restriction of $\pi^*H$ to a fiber is zero. Since $D+\pi^*H$ is ample the restriction is as well. Choose a height function $h_H$ for $H$. We have that \[\hat{h}_D+h_H\circ \phi\] is a height function for $D+\pi^*H$ and the restriction to $F$ is a height function for $(D+\pi^*H)\mid_{F}$. Since the fiber is contracted there is a point $P\in F$ with $\hat{h}_{D}(P)\neq 0$. As otherwise the ample height function $\hat{h}_D+h_H\circ \phi$ is constant, which is absurd as the fiber is not finite. So we have found $P\in F$ with $\hat{h}_D(P)\neq 0$.  Then we compute
	\[h_A(f^n(P))=\lambda^n\hat{h}_D(P)+h_H(Q)\]
	which tells us that $\alpha_f(P)=\lambda$ as desired as $A$ is an ample height function, and the arithmetic degree is independent of the chosen height function.
	\end{proof}
\begin{proposition}\label{prop:birationalbackbone}
	Let $X$ be a projective normal $\QQ$-factorial variety with at worst terminal singularities and $\Alb(X)=0$. Suppose that we have a diagram \[\xymatrix{X\ar[r]^f\ar[d]_\phi & X\ar[d]^\phi\\ Y\ar[r]_g& Y}\]
	where $\phi$ is a birational and extremal contraction. Let $\lambda$ be a potential arithmetic degree of $f$ that is not a potential arithmetic degree of $g$. Suppose that there is a non-trivial nef divisor $D$ with $f^*D=\lambda D$ for linear equivalence. Let $E$ be the exceptional locus of $\phi$ and let $Z=\phi(E)$. Suppose that there is a point $Q\in Z$ with $g(Q)=Q$. Then $\lambda$ is an arithmetic degree. 
\end{proposition}
\begin{proof}
	We may choose an ample divisor $H$ on $Y$ with $A=D+\phi^*H$ being ample on $X$. Set $F=\phi^{-1}(Q)$. Let $\hat{h}_{D}$ be the canonical height function of $D$ 
	\[D+\pi^*H\mid_{F}=D\mid_F\]
	since the restriction of $\pi^*H$ to a fiber is zero. Since $D+\pi^*H$ is ample the restriction to the fiber is as well. As we argued above there is a point $P\in F$ with $\hat{h}_{D}(P)\neq 0$. So we have found $P\in F$ with $\hat{h}_D(P)\neq 0$. As before we compute 
	\[h_A(f^n(P))=\lambda^n\hat{h}_D(P)+h_H(Q)\]
	which tells us that $\alpha_f(P)=\lambda$ as desired. 
	\end{proof}	
The above result tells us that the new eigenvalue introduced by a fibering type contraction is always achieved provided the base morphism has a fixed point and that we have a nef eigendivisor.  
\begin{remark}
	\normalfont One runs into the following problem when trying to run a minimal model type program to obtain realizability results. Suppose that $\mu$ is a potential arithmetic degree, and $1<\vert \mu\vert<\lambda_1(f)$. To find a point $P$ with $\alpha_f(P)=\vert \mu\vert$ we must have that $\overline{\OO_f(P)}=V_P$ is a proper sub variety, and $\alpha_{f\mid_{V_P}}(P)=\vert \mu\vert$. By construction $V_P$ has a dense orbit given by $\OO_f(P)$. Thus the Kawaguchi-Silverman conjecture suggests that $\lambda_1(f\mid_{V_P})=\vert\mu\vert$. So to realize potential arithmetic degrees one must find invariant sub varieties where $\lambda_1(f\mid_V)<\lambda_1(f).$ However in very general situations it seems difficult to find invariant sub varieties. Even in the case of a fibering type contraction there are potential difficulties. Suppose that $\phi\colon X\ra Y$ is a fibering type contraction, and that we have a diagram 
	\[\xymatrix{X\ar[r]^f \ar[d]_\phi & X\ar[d]^\phi\\ Y\ar[r]_g & Y}\]
	Suppose that every potential arithmetic degree of $g$ is realized  and that $\lambda_1(f)\geq\lambda_1(g)$. If say $\alpha_g(Q)=\vert\mu\vert$ then choose a nef eigendivisor $D_\lambda$ for $\lambda$ and an ample divisor $H$ for $Y$ then consider the ample height $h_A=\hat{h}_{D_\lambda}+h_H\circ \phi$  To construct a point $P$ with $\alpha_f(P)=\vert\mu\vert$ we must find a point $P$ such that $\hat{h}_{D_\lambda}(P)=0$ and $\alpha_g(\phi(P))=\vert\mu\vert$. A result of this type seem to require some knowledge of the set of points where the canonical height $\hat{h}_{D_\lambda}$ vanishes. This set is extremely interesting but currently still mysterious.   
\end{remark}
\subsection{Realizability for toric varieties.}
In this section we prove that every potential arithmetic degree of a surjective morphisms of $\QQ$-factorial toric varieties is realizable as an arithmetic degree. We first consider equivariant morphisms in \ref{subsec:toricmorphisms} and give a classification result for such morphisms in \ref{thm:simple=linsimple}. We apply \ref{thm:simple=linsimple} to prove the sAND conjecture for equivariant morphisms of $\QQ$-factorial toric varieties in \ref{thm:toricsAND} and show that any equivariant morphism of $\QQ$-factorial toric varieties has arithmetic eigenvalues in \ref{thm:equivariantrealizability}. We then turn to the case of general morphisms in \ref{subsubsec:nontoricrealizability} and prove that every surjective morphism of $\QQ$-factorial toric varieties has arithmetic eigenvalues. Our strategy will be to realize all the potential arithmetic degrees as the degree on the fiber of an extremal contraction as in \ref{prop:fiberingbackbone} and \ref{prop:birationalbackbone}. Notice that to apply \ref{prop:fiberingbackbone} and \ref{prop:birationalbackbone} one must be able to find fixed points of morphisms on the target of an extremal contraction. Since we may freely iterate our morphism by \ref{prop:iterationprop} we must be able to find pre-periodic points. In this section, the relevant varieties will be $\QQ$-factorial toric varieties, so we seek to guarantee the existence of fixed points for endomorphisms of $\QQ$-factorial toric varieties. In what follows we will use (\cite{MM}) for the minimal model program applied to toric varieties. 
\begin{subsubsection}{Toric morphisms}\label{subsec:toricmorphisms} Projective toric varieties provide an interesting class of varieties all of which admit surjective endomorphisms that are not automorphisms. It seems natural to study the endomorphism schemes $\Send(X_\Sigma)$ for a projective toric variety $X_\Sigma$ and the collection of \emph{equivariant} surjective endomorphisms $\Send_{T_{\Sigma}}(X_\Sigma)$ where $T_\Sigma$ is the dense torus of $\Sigma$. Let $X_\Sigma$ be a projective toric variety $\Sigma\subseteq N\cong \ZZ^r$. Then the toric surjective endomorphism of $X_\Sigma$ correspond to matrices in $\textnormal{GL}(r,\QQ)$ with integer entries that preserve $\Sigma$ is the sense that if $\sigma\in \Sigma$ and $f$ is such a matrix then $f(\sigma)\subseteq \sigma^\prime\in \Sigma$. Since $f_n=n\cdot \textnormal{Id}_r$ preserves all cones when $n>0$ we have that 
	\begin{equation}
		\ZZ_{\geq 0}\subseteq \Send_{T_\Sigma}(X_\Sigma).
	\end{equation}
Associated to any $f\in \Send(X_\Sigma)$ we have a natural linearization anti-homomorphism of monoids which sends $f\mapsto f^*\colon N^1(X_\Sigma)_\QQ\ra N^1(X_\Sigma)_\QQ$. Anti-homomorphism here refers to the fact that $(f\circ g)^*=g^*f^*$. In other words, we have an anti-homomorphism which we call $\textnormal{Lin}$ which linearizes a morphism,
	\begin{equation}
		\textnormal{Lin}\colon \Send(X_\Sigma)\ra \textnormal{GL}(N^1(X)_\QQ),f\mapsto f^*.	
	\end{equation}	
One may check that $n\in \ZZ_{\geq 0}\subseteq \Send(X_\Sigma)$ is mapped to 
	\begin{equation}
		n\cdot \textnormal{Id}_{N^1(X)_\QQ}.
	\end{equation}
	In other words, $\ZZ_{\geq 0}\subseteq \Send(X_\Sigma)$ is mapped to $\ZZ_{\geq 0}\subseteq N^1(X_\Sigma)_\QQ$. This leads to the following natural question.
	\begin{question}\label{ques:toricendoquestion}
		\normalfont	
		Let $X_\Sigma$ be a projective toric variety. 
		\begin{enumerate}
			\item For which toric varieties is $\textnormal{Lin}(\Send(X_\Sigma))$ strictly larger then $\ZZ_{\geq 0}$? In other words, which toric varieties possess a surjective endomorphism which is not polarized?
			\item For which toric varieties is $\textnormal{Lin}(\Send_{T_\Sigma}(X_\Sigma))$ strictly larger then $\ZZ_{\geq 0}$?  In other words, which toric varieties possess a \emph{equivariant} surjective endomorphism which is not polarized?
			\item Can it be the case that  $\textnormal{Lin}(\Send_{T_\Sigma}(X_\Sigma))$ is strictly smaller then $\textnormal{Lin}(\Send(X_\Sigma))$. In other words, is the linear action of surjective endomorphisms of a toric variety completely determined by the linear action of equivariant surjective endomorphisms?
		\end{enumerate}
	\end{question}	
	\begin{definition}\label{def:linearlysimple}
		Let $X_\Sigma$ be a projective toric variety defined over $\Qb$. We say that $X_\Sigma$ is \textbf{linearly simple} if $\textnormal{Lin}(\Send_{T_\Sigma}(X_\Sigma)$ has finite index in $\ZZ_{\geq 0}$. In other words, every surjective toric  morphism is induced by a homomorphism of tori $(x_1,\dots, x_n)\mapsto (x_1^d,\dots, x_n^d)$ for some $d>0$ after possibly iterating the morphism.
	\end{definition}
	We think about this in the following way. A toric variety is linearly simple when the sub-group of obvious surjective endomorphisms is large in the sense that it is finite index. In terms of the action on $N^1(X)_\RR$ we have the following interpretation.
	\begin{proposition}\label{prop:linsimple}
		Let $X_\Sigma$ be a projective toric variety defined over $\Qb$. Then $X_\Sigma$ is linearly simple if and only if for all $f\in\textnormal{Lin}(\Send_{T_\Sigma}(X_\Sigma)$ then eigenvalues of $f^*\colon N^1(X)_\RR\ra N^1(X)_\RR$ have the same magnitude.
	\end{proposition}
	\begin{proof}
		Suppose that $f$ is linearly simple. Then there is some $n\geq 1$ such that $(f^n)^*$ acts on $N^1(X)_\RR$ by scalar multiplication by $\lambda>0$. Thus $(f^n)^*D\Lin \lambda D$ for divisors $D$. If $\mu$ is an eigenvalue of $f^*$ then $\mu^n=\lambda$.  Thus $\vert \mu\vert=\vert \lambda\vert^{\frac{1}{n}}$. Conversely suppose that if $f\in\textnormal{Lin}(\Send_{T_\Sigma}(X_\Sigma))$ then every eigenvalue of $f$ has the same magnitude. Since the nef cone of a projective toric variety is finitely generated (\cite[6.3.20]{CLOToric}). Note that we have that $(f^m)^*$ is diagonalizable with real eigenvalues $\lambda_1,\ldots, \lambda_\rho$. So $(f^{2m})^*$ has positive real eigenvalues which are all of the same magnitude. Thus all eigenvalues are the same, and since $(f^{2m})^*$ is diagonalizable we have that $(f^{2m})^*$ acts by scalar multiplication on $N^1(X)_\RR$ as needed. 
	\end{proof}
	For toric surfaces we already have results that can be leveraged to answer question \ref{ques:toricendoquestion}.
	\begin{theorem}\label{thm:linsimplesurfaces}
		Let $X_\Sigma$ be a smooth projective toric surface.
		\begin{enumerate}
			\item Then $\textnormal{Lin}(\Send(X_\Sigma))$ is linearly simple if $X_\Sigma$ is not isomorphic to $\PP^1\times \PP^1$. Furthermore $\PP^1\times \PP^1$ is not linearly simple. 
			\item For all cases we have that
			\[\textnormal{Lin}(\Send(X_\Sigma))\]
			is also finite index in $\ZZ_{\geq 0}$.
			In other words, the linear part of dynamics on a smooth toric surface is determined by equivariant morphisms.
			\end{enumerate} 
	\end{theorem}
	\begin{proof}
		This follows directly from \ref{prop:TSurfaceTriv}.	
	\end{proof}
Notice that $\PP^n$ is linearly simple for any $n$. We have the following basic necessary condition for $X_\Sigma$ to be linearly simple.
	\begin{proposition}\label{prop:notsimple}
		Let $X_\Sigma$ be a projective toric variety defined over $\Qb$. If the fan $\Sigma$ has a non-trivial decomposition as $\Sigma_1\times \Sigma_2$ then $X_\Sigma$ is not linearly simple.
	\end{proposition}
\begin{proof}
If $X_\Sigma$ and $X_\triangle$ are any two toric varieties then $X_\Sigma\times X_\triangle\cong X_{\Sigma\times \triangle}$ is not linearly simple. To see this note that we may take $f_n\colon X_\Sigma\ra X_\Sigma$ to be the surjective endomorphism induced by multiplication by $n$ on the lattice $N_\Sigma$ containing the fan $\Sigma$ and $g_m\colon X_\triangle\ra X_\triangle$ the equivariant morphism induced by multiplication by $n\neq m$ on the lattice $N_\triangle$ containing the fan $\triangle$. Then $f_n\times g_m\colon X_\Sigma\times X_\triangle\ra X_\Sigma\times X_\triangle$ is a surjective endomorphism that does not act by scalar multiplication on $X_\Sigma$. So $X_\Sigma$ is not linearly simple. 
\end{proof}	
	This leads to the following definition.
	\begin{definition}[Simple toric varieties.]\label{def:simpletoricvar}
		Let $X_\Sigma$ be a $\QQ$-factorial projective toric variety defined over $\Qb$. We say that $X_\Sigma$ is decomposable if
		\[X_\Sigma=X_{\triangle_1}\times X_{\triangle_2}\]
		with each $X_{\triangle_i}$ a $\QQ$-factorial projective toric variety of dimension at least $1$. We say that $X_\Sigma$ is simple if it is not decomposable.
		\end{definition}
	We think of the above definition a toric analogy of the definition of a simple abelian variety, and $\Send_{T_\Sigma}(X_\Sigma)$ an analogy for the endomorphism ring of an abelian variety. The following is an immediate corollary of \ref{thm:linsimplesurfaces}
	\begin{corollary}
		Let $X_\Sigma$ be a smooth projective toric surface defined over $\Qb$. Then $X_\Sigma$ is simple if and only if it is linearly simple.
	\end{corollary}
It is natural to wonder if the analogous result holds for higher dimensional varieties. We now prove that this is indeed the case in \ref{thm:simple=linsimple}.		
	\begin{lemma}\label{lem:iterationlemma}
		Let $X_\Sigma$ be a $\QQ$-factorial projective toric variety defined over $\Qb$. Let $f\colon X_\Sigma\ra X_\Sigma$ a surjective endomorphism induced by a mapping of lattices $\phi\colon N\ra N$. Then $\phi$ is injective and if $\sigma$ is a ray of $\Sigma$ then $\phi(\sigma)$ is a ray of $\Sigma$.
	\end{lemma}
	\begin{proof}
		Because $\phi$ induces a surjective endomorphism of toric varieties, $\phi\colon N\otimes_\ZZ \QQ\ra N\otimes_\ZZ \QQ$ is an isomorphism of rational vector spaces. This is because the induced mapping $\phi^\vee\colon M\ra M$ induces a map of semi group rings
		\[\phi^\vee \colon \Qb[M]\ra \Qb[M]\]
		that induces the homomorphism of tori
		\[\spec \Qb[M]=T_{\Sigma}\ra T_{\Sigma}=\spec \Qb[M]\]
		associated to $f\colon X_\Sigma\ra X_\Sigma$. Since this map is surjective it is dominant and so is $T_{\Sigma}\ra T_{\Sigma}.$ As this is a morphism of affine schemes it is dominant if and only if the morphism of algebras 
		\[\phi^\vee \colon \Qb[M]\ra \Qb[M]\]
		is injective, this occurs precisely when $\phi^\vee\colon M\ra M$ is injective. Thus on the level of vector spaces we have that $\phi^\vee$ and thus $\phi$ is injective and so bijective being a linear mapping between vector spaces of the same dimension. Now let $\sigma$ be a ray of $\Sigma$. Then as $\phi$ is compatible with the fan $\Sigma$ we have that $\phi(\sigma)\subseteq \tau$ where $\tau\in \Sigma$. Suppose that $\tau$ is the minimal such cone. Associated to $\sigma,\tau$ are torus orbits, $\OO(\sigma),\OO(\tau)$ and by \cite[3.3.21]{CLOToric} we have that
		\begin{equation}\label{eq:torusorbitsineq}
			f(\OO(\sigma))\subseteq \OO(\tau).
		\end{equation}
		Suppose that $N$ is $n$-dimensional so that $\dim X_\Sigma=n$. Then if $\gamma\in \Sigma$ is of dimension $k$ we have that $\dim \OO(\gamma)=n-k$. As $\sigma$ is a ray we have that $\dim \OO(\sigma)=n-1$. Since $f$ is a finite morphism we have that $f(\OO(\sigma))$ has dimension $n-1$ as well. So $\dim \OO(\tau)\geq \dim f(\OO(\sigma))=n-1$. So $\OO(\tau)$ has dimension $n-1$ or dimension $n$. If $\OO(\tau)$ has dimension $n$ then $\tau$ must be zero dimensional which is impossible since $\phi$ is injective and so $\phi(\sigma)$ is at least one dimensional. It follows that $\tau$ is of dimension $n-1$ and so $\tau$ is a ray as needed.
	\end{proof}
	
	\begin{lemma}\label{lem:toricmorphismbackbone}
		Let $X_\Sigma$ be a $\QQ$-factorial projective toric variety defined over $\Qb$. Let $X_\Sigma\ra X_\Sigma$ be an equivariant surjective endomorphism induced by a morphism of lattices $f\colon N\ra N$. Then $f^n$ is diagonalizable for some $n$. Let $f^n$ have eigenvalues $\lambda_1,....,\lambda_s$ of multiplicities $m_1,...,m_s$. Let $E_i$ be the $\lambda_i$ eigenspace of $f$. Put 
		\[\Sigma_i=\Sigma\cap E_i=\{\sigma\cap E_i \colon \sigma\in \Sigma\}.\]
		Then $\Sigma_i$ is a complete fan in $E_i$ and we have a decomposition
		\[X_\Sigma=X_{\Sigma_1}\times\dots\times X_{\Sigma_s}.\]
		\end{lemma}
	\begin{proof}
		We first show that $f^n$ is diagonalizable for some $n$. By  \ref{lemma:iterationlemma} $f$ maps rays to rays. Since $f$ is also injective we have that $f$ permutes the rays of $\Sigma$. So for some $m$ we have that $f^m$ fixes the rays of $\Sigma$. So we may replace $f$ with $f^m$ and assume that $f$ fixes all the rays of $\Sigma$. Since there are at least $\dim X=n$ rays we have that $f$ has at least $n$-eigenvectors and so a basis of eigenvectors because there is a maximal dimensional cone with a basis of $\QQ$ eigenvectors. We conclude that $f$ is diagonalizable. We now show that $\Sigma_i$ is a fan for any $i$. Let $\sigma_i=\sigma\cap E_i$ for any $\sigma\in \Sigma$. We may assume that $v_1,...,v_w$ are the ray generators of $\sigma$ and $v_1,...,v_t\in E_i$. Then I claim that
		\begin{equation}
			\sigma_i=\{\sum_{i=1}^ta_iv_i\colon a_i\geq 0\}.
		\end{equation}
		It is clear that $\{\sum_{j=1}^ta_jv_j\colon a_j\geq 0\}\subseteq \sigma_i$. Now let $v\in \sigma_i$. Then $v=\sum_{j=1}^wb_jv_j$. As each $v_j$ is a ray it is an eigenvector and so we may write $f(v_j)=\gamma_jv_j$ where $\gamma_j=\lambda_i$ for $1\leq i\leq t$ and $\gamma_j\neq \lambda_i$ for $j>t$. We then compute
		\begin{align}
			f(v)&=\lambda_iv=\sum_{j=1}^w\lambda_ib_jv_j\\
			&=\sum_{j=1}^wb_jf(v_j)=\sum_{j=1}^t\lambda_ib_jv_j+\sum_{j>t}\gamma_jb_jv_j.
		\end{align}		
		We then have that
		\begin{align}
			0&=\sum_{j=1}^w\lambda_ib_jv_j-(\sum_{j=1}^t\lambda_ib_jv_j+\sum_{j>t}\gamma_jb_jv_j)
			\\&=\sum_{t<j\leq w}(\lambda_i-\gamma_j)b_jv_j.	
		\end{align}	
		Since we assumed that $\Sigma$ was simplicial we have that $v_1,...,v_w$ are independent. So for all $t<j\leq w$ we have that
		\[(\lambda_i-\gamma_j)b_j=0.\]
		Since $\lambda-\gamma_j\neq 0$ by assumption we have $b_j=0$ for $w\geq j>t$. Thus
		\[v=\sum_{j=1}^tb_jv_j\] and $\sigma_i$ is spanned by the rays of $\sigma$ in $E_i$. So $\sigma_i$ is a finitely generated polyhedral cone that is strongly convex as if $\sigma\cap E_i$ contains a non-trivial linear subspace then so does $\sigma$, contradicting that $\sigma\in \Sigma$ is strongly convex. Now if $\tau$ is a face of $\sigma$ then $\tau_i$ is a face of $\sigma_i$ since if $\tau_i=u^\perp\cap \sigma$ with $u\in \sigma^\vee$. Then $u$ restricted to $E_i$ gives an element of the dual space $E_i^*$ say $u_i$. We also have $u_i\in \sigma_i^\vee$ and 
		\[u_i^\perp\cap \sigma_i=\{v\in \sigma_i\colon (u_i,v)=0\}=\{v\in \sigma\cap E_i: (u,v)=0\}=\tau\cap E_i.\]
		So $\tau_i$ is a face of $\sigma_i$. Now we have that since $\Sigma$ is a fan that given any two cones,$\sigma,\sigma^\pp$ that $\sigma\cap \sigma^\prime$ is a face of both $\sigma$ and $\sigma^\prime$. So 
		\[\sigma_i\cap \sigma_i^\pp=(\sigma\cap \sigma^\pp)_i\]
		is a face of $\sigma_i$ and $\sigma_i^\pp$ and so $\Sigma_i$ is a fan. Since $\Sigma$ is complete $\vert\Sigma\vert=N_\RR$ and so $\vert \Sigma\vert \cap E_i=E_i$ and $\Sigma_i$ is a complete fan. Furthermore, given any cone $\sigma$ we have that
		\[\sigma=\bigoplus\sigma_i\]
		since we showed that $\sigma_i$ is generated precisely by the ray generators of $\sigma$ in $E_i$. It follows that we have a decomposition
		\[X_\Sigma\cong X_{\Sigma_1}\times \dots \times X_{\Sigma_s}.\]
		As each $X_{\Sigma_i}$ is a closed sub-variety of $X_\Sigma$ the $X_{\Sigma_i}$ are projective toric varieties. Since given $\sigma$ the ray generators of $\sigma_i$ are a subset of the ray generators of $\sigma$ we have that since $X_\Sigma$ is $\QQ$-factorial that $\sigma_i$ has a linearly independent ray generating set as well. So each $X_{\Sigma_i}$ is a $\QQ$-factorial projective toric variety of dimension $m_i$ and $X_\Sigma$ decomposes as claimed.
		\end{proof}
	We see that the eigenspaces of a surjective toric morphism decompose the toric variety.
	\begin{lemma}\label{lem:toricsk}
	Let $X_\Sigma$ be a $\QQ$-factorial projective toric variety defined over $\Qb$. Suppose that surjective toric morphism $f\colon X_\Sigma\ra X_\Sigma$ is induced by a lattice mapping $\phi\colon N\ra N$. If $\phi$ is scalar multiplication by $n\geq 1$ then $f^*\colon N^1(X_\Sigma)_\RR\ra N^1(X_\Sigma)_\RR$ is scalar multiplication by $n$. 
	\end{lemma}
\begin{proof}
By \cite[4.2.8]{CLOToric} a Cartier divisor $D$ on a toric variety $X_\Sigma$ is equivalent to a collection $(m_\sigma)_{\sigma\in \Sigma}$ where $m_\sigma\in M$ and $D$ on $U_\sigma$ has local equation $\chi^{-m_\sigma}$. To pullback $D$ we pull back the local equations and obtain that $f^*D$ is has local equation $f^*\chi^{-m_\sigma}=\chi^{-\phi^\vee(m_\sigma)}$. Since $\phi$ is multiplication by a scalar $n$ we have that $\phi$ is represented by a matrix of the form $n\cdot \textnormal{I}_{\dim X_\Sigma}$. Then $\phi^\vee$ is given by the transpose  $(n\cdot \textnormal{I}_{\dim X_\Sigma})^t= n\cdot \textnormal{I}_{\dim X_\Sigma}$. Thus $\phi^\vee (m_\sigma)=nm_\sigma$. Consequently we have that $f^*D\Lin nD$ as required. 

\end{proof}
	
	\begin{theorem}\label{thm:simple=linsimple}
		Let $X_\Sigma$ be a $\QQ$-factorial projective toric variety defined over $\Qb$. Then $X_\Sigma$ is linearly simple if and only if $X_\Sigma$ is simple.
	\end{theorem}
	\begin{proof}
		Suppose that $X_\Sigma$ is linearly simple. Then $X_\Sigma$ is simple by \ref{prop:notsimple}. Now suppose that $X_\Sigma$ is simple. Towards a contradiction let $f\colon X_\Sigma\ra X_\Sigma$ be an equivariant surjective endomorphism with $f$ not linearly simple. Let $f$ be induced by a lattice map $\phi$. After iterating $\phi$ we may assume that $\phi$ fixes the rays of $\Sigma$ and diagonalizable by \ref{lem:toricmorphismbackbone}. Since $f$ is not linearly simple we have by \ref{lem:toricsk} we have that $\phi$ is not multiplication $n\geq 1$. Then $\phi$ must have at least two distinct eigenvalues. By \ref{lem:toricmorphismbackbone} we have that $X_\Sigma$ decomposes non-trivially contradicting our assumption.  
	\end{proof}
	It is natural at this point to ask, to what extent is a decomposition 
	\[X_\Sigma\cong X_{\Sigma_1}\times \dots \times X_{\Sigma_r}\]
	into simple toric varieties unique. We intend to return to this issue in the future. The ultimate goal being some sort of \emph{analogy} between dynamics in the toric situation and dynamics in the abelian variety situation. These results can now be applied to the dynamics of toric morphisms. We first give a new proof of the sAND conjecture for equivariant surjective toric morphisms. 
	\begin{theorem}\label{thm:toricsAND}
		Let $X_\Sigma$ be a $\QQ$-factorial toric variety defined over $\Qb$ and $f\colon X_\Sigma\ra X_\Sigma$ an equivariant surjective toric morphism. Then the sAND conjecture holds for $f$.
	\end{theorem}
	\begin{proof}
		Suppose that $f$ is induced by a lattice mapping $\phi$. By \ref{lem:iterationlemma} we may assume that $\phi^m$ fixes the rays of $X_\Sigma$. Now write
		\begin{equation}
			X_{\Sigma_1}\times\dots\times X_{\Sigma_r}
		\end{equation}
		where each $X_{\Sigma_i}$ is simple. Since $\phi^m$ fixes the rays of $\Sigma$ we have that $f^m=h_1\times\dots \times h_r$ where $h_i\colon X_{\Sigma_i}\ra X_{\Sigma_i}$ is a surjective equivariant endomorphism. Note that
		\[\lambda_1(f^m)=\max_{i=1}^r\{\lambda_1(h_i)\}.\]
		We may assume that $\lambda_1(f)>1$ as the sAND conjecture is trivial when $\lambda_1(f)=1$. For some $i$ we have $\lambda_1(f^m)=\lambda_1(h_i)>1$. Let $\pi_i\colon X_\Sigma\ra X_{\Sigma_i}$ be the canonical projection. Since $X_{\Sigma_i}$ is simple it is linearly simple by \ref{thm:simple=linsimple}. As $\lambda_1(h_i)>1$ we have that $h_i^*\colon N^1(X_{\Sigma_i})_\RR\ra N^1(X_{\Sigma_i})_\RR $ is multiplication by $\lambda_1(h_i)$. Now fix a number field $K$ over which our data is defined and choose $d\geq 1$. The sAND conjecture is equivalent to the assertion that
		\begin{equation}
			\McS_{K,d}=\{P\in X(\Qb)\colon [K(P):K]\leq d, \alpha_f(P)<\lambda_1(f)\}
		\end{equation} 
		is not Zariski dense. Note that $h_i\colon X_{\Sigma_i}\ra X_{\Sigma_i}$ is a surjective toric morphism and $X_{\Sigma_i}$ is simple. Therefore by \ref{thm:simple=linsimple} $h_i$ is linearly simple. As $1<\lambda_1(f^m)=\lambda_1(h_i)$ we may assume that $h_i$ is polarized, in particular there is an ample divisor $H_i$ on $X_{\Sigma_i}$ with $h_i^*H_i\Lin\lambda_1(h_i)H_i$ Thus $\alpha_{h_i}(P^\pp)=\lambda_1(h_i)$ unless the canonical height $\hat{h}_{H_i}(P^\pp)=0$. By the Northcott property for $\hat{h}_{H_i}$ there are finitely many points $P^\pp\in X_{\Sigma_i}(\Qb)$ with $\hat{h}_{H_i}(P^\pp)=0$ and $[K:K(P)]\leq d$. Let $Q_{i1},...,Q_{is}\in X_{\Sigma_i}(\Qb)$ be the points with vanishing canonical height just described and residue degree at most $d$ that was just described. Note that we have 
		\[\alpha_{f^m}(P)=\max_{i=1}^r\{\alpha_{h_i}(\pi_i(P))\}.\]
		Therefore, we have that $\alpha_{f^m}(P)=\lambda_1(h_i)=\lambda_1(f^m)$ except possibly on the proper Zariski closed set \[\bigcup_{j=1}^{s}\pi_i^{-1}(Q_{ij}).\]
		So we have that
		\[\alpha_f(P)^m=\alpha_{f^m}(P)=\lambda_1(f^m)=\lambda_1(f)^m\]
		except at 
		\[\bigcup_{j=1}^s\pi_i^{-1}(Q_{ij}).\]
		Taking $m^{th}$ roots now gives the desired result. 
		
	\end{proof}
	We now turn to realizability.
	\begin{theorem}\label{thm:equivariantrealizability}
		Let $X_\Sigma$ be a $\QQ$-factorial toric variety defined over $\Qb$ and $f\colon X_\Sigma\ra X_\Sigma$ an equivariant surjective toric morphism. Then $f$ has arithmetic eigenvalues.
	\end{theorem}
	\begin{proof}
		Suppose that $f$ is induced by a lattice mapping $\phi$. By \ref{lem:iterationlemma} we may assume that $\phi^m$ fixes the rays of $X_\Sigma$. Now write
		\begin{equation}
			X_{\Sigma_1}\times\dots\times X_{\Sigma_r}
		\end{equation}
		where each $X_{\Sigma_i}$ is simple. Since $\phi^m$ fixes the rays of $\Sigma$ we have that $f^m=h_1\times\dots \times h_r$ where $h_i\colon X_{\Sigma_i}\ra X_{\Sigma_i}$ is a surjective equivariant endomorphism. By \ref{prop:iterationprop} we may replace $f$ with $f^m$ and prove the result. We may assume that $f=h_1\times\dots\times h_r$. If $\lambda_1(f)=1$ then there is nothing to prove. Otherwise assume that $\lambda_1(f)>1$. Note that  the Picard number of $X_\Sigma$ is $d-n$ where $d$ is the number of rays in $\Sigma$. Let $d_i$ be the number of rays in $\Sigma_i$ and $n_i$ the dimension of $X_{\Sigma_i}$. Now note that $\pi_i^*\colon N^1(X_{\Sigma_i})_\RR\ra N^1(X_\Sigma)_\RR$ is an injection and so the image has dimension $d_i-n_i$. We have that $\bigoplus_{i=1}^r \pi_i^*N^1(X_{\Sigma_i})_\RR\subseteq N^1(X_\Sigma)_\RR$ has rank
		\[\sum_{i=1}^r(d_i-n_i)=d-n\]
		as $\sum_{i=1}^rd_i=d$ and $\sum_{i=1}^rn_i=n=\dim X_\Sigma$.  Thus 
		\[N^1(X_\Sigma)_\RR=\bigoplus_{i=1}^r \pi_i^*N^1(X_{\Sigma_i})_\RR\]
		and the action of $f$ on $N^1(X_\Sigma)_\RR$ is given by
		\[f^*\sum_{i=1}^rD_i=\sum_{i=1}^rh_i^*D_i\]
		where $D_i\in \pi_i^*N^1(X_{\Sigma_i})_\RR$. It follows that the only eigenvalues for $f^*$ are the eigenvalues of the various $h_i^*$. This means that the eigenvalues of $f^*$ are precisely the integers $n_i=\lambda_1(h_i)$. This is because $h_i\colon X_{\Sigma_i}\ra X_{\Sigma_i}$ is a surjective endomorphism with $ X_{\Sigma_i}$ simple, by \ref{thm:simple=linsimple} we have that $h_i^*\colon N^1(X_{\Sigma_i})_\RR\ra N^1(X_{\Sigma_i})_\RR$ acts by multiplication by $n_i=\lambda_1(h_i)$. Now consider any $n_i>1$. Choose $P\in X_{\Sigma_i}(\Qb)$ with $\alpha_{h_i}(P)=n_i$. Let $e_j$ be the identity of the torus for $X_{\Sigma_j}$ Let $Q$ be the point of $X_\Sigma$ whose $i^{th}$ coordinate is $P$ and for $j\neq i$ the $j^{th}$ coordinate is $e_j$. In other words, $Q$ is a point with $\pi_i(Q)=P$ and $\pi_j(Q)=e_j$ for $j\neq i$. Note that $h_j\colon X_{\Sigma_j}\ra X_{\Sigma_j}$ is induced by a lattice homomorphism that is multiplication by a scalar. So $h_j(t_1,...,t_s)=(t_1^{n_j},...,t_s^{n_j})$ on the torus of $X_{\Sigma_j}$. Thus $h_j(e_j)=e_j$. Then we have that $\alpha_{h_j}(\pi_j(Q))=1$ for $j\neq i$ and $\alpha_{h_i}(\pi_i(Q))=n_i$. Thus, we have that 
		\[\alpha_f(P)=\max_{j=1}^r\{\alpha_{h_j}(\pi_j(P))\}=n_i\]
		as needed. 
	\end{proof}
	In conclusion, equivariant toric surjective morphisms are built out of polarized morphisms of simple $\QQ$-factorial toric varieties. This is reminiscent of the program to understand surjective endomorphisms of projective varieties admitting int-amplified endomorphisms. Notice that a polarized endomorphism is in fact int-amplified, so we have realized this part of the program for this special well behaved class of endomorphisms. 
	
\end{subsubsection}

\begin{subsubsection}{Non-equivariant morphisms}\label{subsubsec:nontoricrealizability}

We now turn to the general case of non-equivariant surjective endomorphisms of $\QQ$-factorial toric varieties. 
\begin{proposition}\label{prop:basictoric1}
	Let $X_\Sigma$ be a $\QQ$-factorial projective toric variety with fan $\Sigma\subseteq N$ and $\tau\in \Sigma$. Then $V(\tau)$ is $\QQ$-factorial and $\rho(V(\tau))=\rho(X_\Sigma)$. 
\end{proposition}
\begin{proof}
	$V(\tau)$ is the toric variety with fan $\textnormal{Star}(\tau)$. Let $v_1,...,v_d$ be the rays of $\Sigma$. After reordering we have that $v_1,...,v_t$ are the rays of $\tau$. Given a cone $\sigma=\textnormal{Cone}(v_1,...,v_t,v_{t+1},...,v_s)$ with face $\tau$ the associated cone in $\textnormal{Star}(\tau)$ is given by $\textnormal{Cone}(\bar{v}_{t+1},...,\bar{v}_s)$. Now suppose that $\bar{v}_{t+1},...,\bar{v}_s$ was not independent. Then we could find scalars not all zero with 
\[a_{t+1}\bar{v}_{t+1}+...+a_s\bar{v}_s=0\]
which means we can find scalars $a_1,\ldots,a_t$ with
\[a_{t+1}v_{t+1}+...+a_sv_s=a_1v_1+\ldots+a_tv_t\]
which contradicts $\sigma$ being a simplicial cone. Thus $V(\tau)$ is simplicial. The rays of $V(\tau)$ are then  $\bar{v}_{t+1},...,\bar{v}_d$. Since $V(\tau)$ has dimension $n-t$ and $d=n+\rho$ we have that there are $d-t=n+\rho-t=n-t+\rho$ rays. So $\rho(V(\sigma))=\rho(X_\Sigma)$ as desired. 
\end{proof}
\begin{lemma}\label{lem:toricpre1}
	Let $X$ be a $\QQ$-factorial projective toric variety of Picard number 1. Let $f\colon X\ra X$ be a surjective endomorphism. Then $f$ has a pre-periodic point.
\end{lemma}
\begin{proof}
	Suppose that $\lambda_1(f)>1$. Then by a result of Fakkruddin \cite[Theorem 5.1]{MR1995861} we have that the set of pre-periodic points is dense in $X$. So we may assume that $\lambda_1(f)=1$ and that $f$ is an automorphism. We now induct on $\dim X=d$. When $d=1$ we have that $X=\PP^1$ and an automorphisms of $\PP^n$ always has a fixed point given by an eigenvector for an associated matrix. Now let $d>1$. First suppose that $X$ is singular. Let $S$ be the singular locus of $X$. Recall that we may write \[S=\bigcup_{\sigma\in I}V(\sigma)\] where $I$ is the set of singular cones of the fan of $\Sigma$. The minimal singular cones thus are the components of the singular locus $S$ and $f$ permutes them being an automorphism. After iterating $f$ we may assume that $f$ fixes the components $V(\sigma)$ and thus we obtain $f^k\colon V(\sigma)\ra V(\sigma)$. Now $V(\sigma)$ is a torus closure of a $\QQ$-factorial toric variety and so is $\QQ$-factorial and of Picard number 1 by Proposition \ref{prop:basictoric1}. So by induction $f^k$ has a pre-periodic point and therefore so does $f$. Now assume that $X$ is smooth. Since $\rho(X)=1$ we have that $X=\PP^n$ (all smooth toric varieties of Picard number 1 are projective spaces) for some $n$ and as noted above every automorphism of $\PP^n$ has a fixed point. 
\end{proof}

\begin{lemma}\label{lem:toricpre2}
	Let $X$ be a $\QQ$-factorial projective toric variety. Let $f\colon X\ra X$ be a surjective endomorphism. Then $f$ has a pre-periodic point.
\end{lemma}
\begin{proof}
	We induct on the dimension. If $\dim X=1$ the result follows from the result on $\PP^1$. Otherwise first suppose that $X$ is not a Mori-fiber space. Then $X$ admits $K_X$ negative birational extremal contraction. Let $E$ be the exceptional locus, then $\dim E<\dim X$ and we have $f\colon E\ra E$ after possibly iterating $f$. Since $E$ is an orbit closure, by Proposition \ref{prop:basictoric1} that $E$ is a $\QQ$-factorial toric variety. By induction we have that $f\colon E\ra E$ has a fixed point and we are done. Otherwise we may assume that $X$ is a Mori-fiber space $\phi\colon X\ra Y$. After iterating $f$ we have a diagram
	\[\xymatrix{X\ar[r]^f\ar[d]_\phi & X\ar[d]^\phi\\ Y\ar[r]_g & Y}\]
     By induction $g\colon Y\ra Y$ has a pre-periodic point. After iterating $f$ and $g$ we may assume that there is a point $Q$ such that $g(Q)=Q$. Then we obtain a mapping $f\colon F\ra F$ where $F=\phi^{-1}(Q)$. Here $F$ is a $\QQ$-factorial toric variety by \cite[15.4.5]{CLOToric}. If $\dim F<\dim X$ then by induction there is a pre-periodic point as needed. This is because $F$ is a normal projective variety, and the fibers of a Mori-fiber space are connected thus $F$ is irreducible. Now $f\colon F\ra F$ is finite mapping as $f$ is a finite mapping. So the image of $f$ is a $\dim F$-dimensional closed sub-variety, it follows that $f$ is surjective so we may apply the inductive hypothesis.  Otherwise $\dim F=\dim X$ and $Y$ is a point. Then $X$ has Picard number $1$ and Lemma \ref{lem:toricpre1}) gives the result.   	
\end{proof}	
\begin{theorem}\label{thm:toricrealworks}
	Let $X$ be $\QQ$-factorial toric variety defined over $\overline{\QQ}$. Let $f\colon X\ra X$ be a surjective endomorphism. Then $f$ has arithmetic eigenvalues. 
\end{theorem}
\begin{proof}
	Let $\lambda$ be a potential arithmetic degree of $f$. After replacing $f$ with an iterate we may assume that $f^*$ fixes all rays of the nef cone of $X$. Let $D_\lambda$ be a nef eigendivisor for $\lambda$. Choose a facet $F$ of $\Nef(X)$ that does not contain $D_\lambda$ and let $\phi\colon X\ra Y$ be the associated extremal contraction. First suppose that $\phi$ is birational. Since a toric variety admits an int-amplified endomorphism, then after iterating $f$ by (\cite[Theorem 5.3]{MR4070310}) we have a conjugating diagram
    \[\xymatrix{X\ar[r]^f\ar[d]_\phi & X\ar[d]^\phi \\ Y\ar[r]_g & Y}\]
    Then we have that $F$ is identified with the Nef cone of $Y$. Let $H$ be an ample divisor on $Y$ such that $A=D_\lambda+\phi^*H$ is ample. This occurs as $D_\lambda$ does not appear in $\phi^*N^1(Y)_\RR$ by construction. Let $E$ be the exceptional locus of $\phi$. Let $Z=\phi(E)$. Notice that we obtain maps $f\colon E\ra E$ and $g\colon Z\ra Z$.  Using $\ref{lem:toricpre2}$ we may take $P$ to be a fixed point of $f$ in $E$ after potentially iterating $f$. Thus we have that $P=\phi(Q)$ is a fixed point of $g$ in $Z$. By the argument in Proposition \ref{prop:birationalbackbone} we have that $\lambda$ is realized.  Now suppose that $\phi$ is fibering. By $\ref{lem:toricpre2}$ we have that $g\colon Y\ra Y$ has a fixed point (after potentially iterating all morphisms) and so using the same argument as the previous paragraph and Proposition \ref{prop:fiberingbackbone} we have that $\lambda$ is realizable. 
\end{proof}
The key result here is two fold: 
\begin{enumerate}
	\item We could find pre-period points to construct good fibers.
	\item we could contract all faces of the nef cone, and that every such contraction corresponds to an extremal contraction of a pair $(X,D)$ which allowed us to conclude the result.
\end{enumerate}
The basic $\emph{reason}$ why the theorem is true, is that potential arithmetic degrees on the fibers can be realized, and in the toric case every potential arithmetic degree appears on the fibers of some extremal contraction. Notice that this shares many details with the argument for surjective equivariant morphisms.  
\end{subsubsection}
\end{section} 

\begin{section}{Realizability in the Int-amplified setting}\label{sec:Intamp1}

In this section we study the realizability question in the setting of varieties admitting int-amplified endomorphisms. We first give some basic results that illustrate the points of friction in this approach.

\begin{proposition}\label{prop:intprop1}
	Let $X$ be a $\QQ$-factorial variety with terminal singularities and finitely generated nef cone and $\Alb(X)=0$. Suppose that $f\colon X\ra X$ is a surjective endomorphism. Let $\phi \colon X\ra Y$ be a birational extremal contraction. Suppose that we have a diagram \[\xymatrix{X\ar[r]^f\ar[d]_\phi & X\ar[d]^\phi \\ Y\ar[r]_g & Y}\]
	Let $E$ be the exceptional locus of $\phi$ and let $Z=\phi(E)$. We have a second diagram given by
	\[\xymatrix{E\ar[r]^f\ar[d]_\phi & E\ar[d]^\phi \\Z\ar[r]_g & Z}\] 
	Suppose that every potential arithmetic degree of $f\mid_E\colon E\ra E$ is realized as an arithmetic degree. If every potential arithmetic degree of $g$ is realizable as an arithmetic degree and $g\colon Z\ra Z$ admits a pre-periodic point. then every potential arithmetic degree of $f$ is realizable as an arithmetic degree. 
\end{proposition}
\begin{proof}
	After iterating $f$ we may assume that $f^*$ fixes the rays of the nef cone and so every eigenvalue has a nef eigendivisor. Let $\mu$ be realizable as a potential arithmetic degree of $g$. Choose a point $P$ so that $\alpha_g(P)=\vert \mu\vert$. If $g^n(P)\notin Z$ for all $n$ then by the birational invariance of arithmetic degrees (\cite[Lemma 2.4]{1802.07388}) we have that $\vert \mu\vert= \alpha_g(P)=\alpha_f(\phi^{-1}(P))$ as needed. So we may assume that $P\in Z$. Therefore $\mu$ is a potential arithmetic degree of $g\mid_Z$ and so a potential arithmetic degree of $f\mid_E$ which by assumption is realizable. Now let $\lambda$ be a potential arithmetic degree of $f$ that is not a potential arithmetic degree of $g$. Then by Proposition \ref{prop:birationalbackbone} we have that $\lambda$ is an arithmetic degree. 
\end{proof}

\begin{corollary}
	In the situation of the above proposition Proposition \ref{prop:intprop1} if $f\colon E\ra E$ is such that $f^*$ acts by a dilation on $N^1(E)_\RR$ and every potential arithmetic degree of $g\colon Y\ra Y$ is an arithmetic degree then every potential arithmetic degree of $f$ is realized.
\end{corollary}
\begin{proof}
	This is immediate as if $f\colon E\ra E$ is polarized, then every potential arithmetic degree realizable, as there is a potential arithmetic degree.	
\end{proof}

Because of the importance we give these eigenvalues a name.

\begin{definition}
	Let $f\colon X\ra X$ be a surjective endomorphism and let $\phi\colon X\ra Y$ be a birational morphism. Let $g\colon Y\ra Y$ be a surjective endomorphism with 
	\[\xymatrix{X\ar[r]^f\ar[d]_\phi & X\ar[d]^\phi \\ Y\ar[r]_g & Y}\]
	Let $E$ be the exceptional locus of $\phi$ and let $Z=\phi(E)$. Suppose we have a second diagram
	\[\xymatrix{E\ar[r]^f\ar[d]_\phi & E\ar[d]^\phi \\Z\ar[r]_g & Z}\] 
	We call the eigenvalues of $f\colon E\ra E$ the exceptional eigenvalues of $f$ with respect to $\phi$.
\end{definition}
The above results say that if we know what happens with the exceptional eigenvalues, then the problem of realizability is translated along a birational extremal contraction. If we imagine attempting to run the minimal model program to simplify the situation, this would correspond to a divisorial contraction. To do the minimal model program one must also consider the situation of flips. This is where we use \cite[Theorem 5.3]{1908.11537} to transfer dynamical information between the flips.
\begin{lemma}
	Let $X$ be $\QQ$-factorial projective variety with at worst terminal singularties and finitely generated nef cone. Suppose that $X$ that admits an int-amplified endomorphism and $\Alb(X)=0$. Take $f\colon X\ra X$ a surjective morphism and let $\phi\colon X\ra Y$ be a flipping extremal contraction with $\phi^+\colon X^+\ra Y$ the associated flip. Let $g\colon Y\ra Y$ be a surjective endomorphism with $\phi\circ f=g\circ\phi$. We let $E$ be the exceptional locus of $\phi$. In addition assume the following.
	\begin{enumerate}
		\item Every potential arithmetic degree of $f\colon E\ra E$ is an arithmetic degree.
		\item Put $Z=\phi(E)$. We assume that $g\colon Z\ra Z$ and that this morphism has a pre-periodic point. 
		\item We assume that every potential arithmetic degree on $X^+$ is an arithmetic degree. 
	\end{enumerate}
	Then every potential arithmetic degree of $X$ is realized as an arithmetic degree.
	
\end{lemma}
\begin{proof}
	Let $\psi\colon X\dashrightarrow X^+$ be the associated birational map. By (\cite[Theorem 5.3]{1908.11537}) we have a diagram
	\[\xymatrix{X\ar[r]^\psi\ar[d]_f & X^+\ar[d]^{f^+}\\ X\ar[r]_\psi & X^+}\]
	where $f^+$ is a everywhere defined morphism.  Since $E$ has codimension two we note that $f^*$ and $(f^+)^*$ have eigenvalues of the same magnitude, so the potential arithmetic degrees of both $f$ and $f^+$ coincide. Let $\lambda$ be a potential arithmetic degree of $f$ and suppose that there is a point $P^+\in X^+$ with $\alpha_{f^+}(P^+)=\vert \lambda\vert $. If $(f^+)^n(P^+)\notin E^+$ for all $n$ then by the birational invariance of the arithmetic degree (\cite[Lemma 2.4]{1802.07388}) we have that $\alpha_f(\psi^{-1}(P^+))=\vert \lambda\vert $ as needed.  We may assume that for all $P\in X^+\setminus E^+$ that $\alpha_f(P)\neq\vert \lambda\vert$. So we may assume that $P^+\in E^+$ and so $\lambda$ is a potential arithmetic degree of $f^+\colon E^+\ra E^+$. If $\lambda$ is a potential arithmetic degree of $g\colon Z\ra Z$ then we are done by assumption because then $\lambda$ is a potential arithmetic degree of $E$. Therefore $\lambda$ is not a potential arithmetic degree of $g$. By Proposition \ref{prop:birationalbackbone} we have that $\lambda$ is realizable as an arithmetic degree. 
	\end{proof}
The above result shows that while flipping the exceptional eigenvalues remain the same, and that they are the issue that prohibits a reduction to mori fiber spaces. We now give an example of how these ideas can be used in practice to prove realizability results. We first treat the case of rationally connected surfaces with $\textnormal{Alb}(X)=0$ that admit an int amplified endomorphism. 
\begin{theorem}\label{thm:mainintamp1}
	Let $X$ be a normal $\QQ$-factorial surface with at worst terminal singularities and finitely generated nef cone that is rationally connected over $\overline{\QQ}$ and $\Alb(X)=0$ that admits an int-amplified endomorphism. Let $f\colon X\ra X$ be a surjective endomorphism and $\lambda$ a potential arithmetic degree of $f$. Then $\lambda$ is realizable as an arithmetic degree.  
\end{theorem}
\begin{proof}
	We proceed by induction on the Picard number $\rho=\rho(X)$. When $\rho=1$ the result is true. When $\rho>1$ suppose that $X$ admits a $K_X$-negative extremal contraction that is divisorial. After iterating $f$  (\cite[Theorem 5.3]{1908.11537}) we have a diagram
	\[\xymatrix{X\ar[r]_f\ar[d]_\phi & X\ar[d]^\phi\\ Y\ar[r]_g & Y }\]
	The exceptional locus is now an irreducible curve $E$ that is contracted by $\phi$ and $Z=\phi(E)$ is a point say $Q$ that is fixed by $g$. Thus by Proposition \ref{prop:birationalbackbone} we have that any potential arithmetic degree of $f$ that is not a potential arithmetic degree of $g$ is realized. By induction every potential arithmetic degree of $g$ is realized and the birational invariance of the arithmetic degree gives the required result. So we may suppose that every $K_X$-negative extremal contraction of $X$ is of fibering type as flips do not occur for surfaces because of codimension reasons. We may now assume the existence of a diagram
	\[\xymatrix{X\ar[r]_f\ar[d]_\phi & X\ar[d]^\phi\\ Y\ar[r]_g & Y }\]  
	where $\phi$ is fibering and $Y$ is a curve. By assumption we have no non-trivial morphism to an elliptic curve or an abelian variety. So we must have that $Y$ is a $\QQ$-factorial and normal curve that is not of general type or an elliptic curve. It follows that $Y=\PP^1$ as a curve of general type admits a morphism to its Jacobian.  The Picard number of $X$ is now $2$. If $\lambda_1(f)=\lambda_1(g)$ then $\lambda_1(f)$ can be realized as an arithmetic degree by Theorem \ref{theorem:MSS}. On the other hand, if $\mu$ is a second potential arithmetic degree in this situation with $\vert \mu\vert <\lambda_1(f)$ then as $g$ has a fixed point, being an endomorphism of $\PP^1$ and $\mu$ is realizable by $\ref{prop:fiberingbackbone}$. Thus we may assume that $\lambda_1(f)>\lambda_1(g)$. In this case \cite[Theorem 5.2]{1908.01605} gives the existence of a morphism $\psi\colon X\ra \PP^1$ and a conjugating morphism $h\colon \PP^1\ra \PP^1$ such that 
	\[\xymatrix{X\ar[r]_f\ar[d]_\psi & X\ar[d]^\psi\\ \PP^1\ar[r]_h & \PP^1 }\] 
	commutes and $\lambda_1(h)=\lambda_1(f)$. The argument just given shows that any other potential arithmetic degree is realizable. 
\end{proof}
\end{section}

\bibliographystyle{plain}
\bibliography{bib}

\begin{thebibliography}{10}

\bibitem{MR1503260}
A.~Adrian Albert.
\newblock Involutorial simple algebras and real {R}iemann matrices.
\newblock {\em Ann. of Math. (2)}, 36(4):886--964, 1935.

\bibitem{Bellon1999}
M.~P. Bellon and C.-M. Viallet.
\newblock Algebraic entropy.
\newblock {\em Communications in Mathematical Physics}, 204(2):425--437, July
  1999.

\bibitem{MR1255693}
Gregory~S. Call and Joseph~H. Silverman.
\newblock Canonical heights on varieties with morphisms.
\newblock {\em Compositio Math.}, 89(2):163--205, 1993.

\bibitem{MR3289919}
Serge Cantat.
\newblock Dynamics of automorphisms of compact complex surfaces.
\newblock In {\em Frontiers in complex dynamics}, volume~51 of {\em Princeton
  Math. Ser.}, pages 463--514. Princeton Univ. Press, Princeton, NJ, 2014.

\bibitem{CLOToric}
David~A. Cox, John~B. Little, and Henry~K. Schenck.
\newblock {\em Toric varieties}, volume 124 of {\em Graduate Studies in
  Mathematics}.
\newblock American Mathematical Society, Providence, RI, 2011.

\bibitem{MR1995861}
Najmuddin Fakhruddin.
\newblock Questions on self maps of algebraic varieties.
\newblock {\em J. Ramanujan Math. Soc.}, 18(2):109--122, 2003.

\bibitem{MR2753603}
Charles Favre and Mattias Jonsson.
\newblock Dynamical compactifications of {${\bf C}^2$}.
\newblock {\em Ann. of Math. (2)}, 173(1):211--248, 2011.

\bibitem{MR2154100}
Yoshio Fujimoto and Noboru Nakayama.
\newblock Compact complex surfaces admitting non-trivial surjective
  endomorphisms.
\newblock {\em Tohoku Math. J. (2)}, 57(3):395--426, 2005.

\bibitem{hindry2013diophantine}
Marc Hindry and Joseph~H. Silverman.
\newblock {\em Diophantine geometry}, volume 201 of {\em Graduate Texts in
  Mathematics}.
\newblock Springer-Verlag, New York, 2000.

\bibitem{AbelianPicardHulek}
Klaus Hulek and Roberto Laface.
\newblock On the {P}icard numbers of {A}belian varieties.
\newblock {\em Ann. Sc. Norm. Super. Pisa Cl. Sci. (5)}, 19(3):1199--1224,
  2019.

\bibitem{MR3189467}
Shu Kawaguchi and Joseph~H. Silverman.
\newblock Examples of dynamical degree equals arithmetic degree.
\newblock {\em Michigan Math. J.}, 63(1):41--63, 2014.

\bibitem{MR4068299}
Shu Kawaguchi and Joseph~H. Silverman.
\newblock Addendum to ``{D}ynamical canonical heights for {J}ordan blocks,
  arithmetic degrees of orbits, and nef canonical heights on abelian
  varieties".
\newblock {\em Trans. Amer. Math. Soc.}, 373(3):2253, 2020.

\bibitem{MR4080251}
Shu Kawaguchi and Joseph~H. Silverman.
\newblock Erratum to: ``{O}n the dynamical and arithmetic degrees of rational
  self-maps of algebraic varieties" ({J}. {R}eine {A}ngew. {M}ath. 713 (2016),
  21--48).
\newblock {\em J. Reine Angew. Math.}, 761:291--292, 2020.

\bibitem{1802.07388}
John Lesieutre and Matthew Satriano.
\newblock Canonical {H}eights on {H}yper-{K}\"{a}hler {V}arieties and the
  {K}awaguchi--{S}ilverman {C}onjecture.
\newblock {\em Int. Math. Res. Not. IMRN}, (10):7677--7714, 2021.

\bibitem{MM}
Kenji Matsuki.
\newblock {\em Introduction to the {M}ori program}.
\newblock Universitext. Springer-Verlag, New York, 2002.

\bibitem{MR4070310}
Yohsuke Matsuzawa.
\newblock Kawaguchi-{S}ilverman conjecture for endomorphisms on several classes
  of varieties.
\newblock {\em Adv. Math.}, 366:107086, 26, 2020.

\bibitem{2002.10976}
Yohsuke Matsuzawa, Sheng Meng, Takahiro Shibata, and De-Qi Zhang.
\newblock Non-density of points of small arithmetic degrees, 2020.

\bibitem{MR3871505}
Yohsuke Matsuzawa, Kaoru Sano, and Takahiro Shibata.
\newblock Arithmetic degrees and dynamical degrees of endomorphisms on
  surfaces.
\newblock {\em Algebra Number Theory}, 12(7):1635--1657, 2018.

\bibitem{1908.11537}
Yohsuke Matsuzawa and Shou Yoshikawa.
\newblock Kawaguchi-silverman conjecture for endomorphisms on rationally
  connected varieties admitting an int-amplified endomorphism, 2019.

\bibitem{MR4074056}
Sheng Meng.
\newblock Building blocks of amplified endomorphisms of normal projective
  varieties.
\newblock {\em Math. Z.}, 294(3-4):1727--1747, 2020.

\bibitem{MR3742591}
Sheng Meng and De-Qi Zhang.
\newblock Building blocks of polarized endomorphisms of normal projective
  varieties.
\newblock {\em Adv. Math.}, 325:243--273, 2018.

\bibitem{1908.01605}
Sheng Meng and De-Qi Zhang.
\newblock Kawaguchi-silverman conjecture for surjective endomorphisms, 2019.

\bibitem{MR4117085}
Sheng Meng and De-Qi Zhang.
\newblock Semi-group structure of all endomorphisms of a projective variety
  admitting a polarized endomorphism.
\newblock {\em Math. Res. Lett.}, 27(2):523--549, 2020.

\bibitem{MumfordAV}
David Mumford.
\newblock {\em Abelian varieties}, volume~5 of {\em Tata Institute of
  Fundamental Research Studies in Mathematics}.
\newblock Published for the Tata Institute of Fundamental Research, Bombay; by
  Hindustan Book Agency, New Delhi, 2008.
\newblock With appendices by C. P. Ramanujam and Yuri Manin, Corrected reprint
  of the second (1974) edition.

\bibitem{MR949273}
F.~Oort and M.~van~der Put.
\newblock A construction of an abelian variety with a given endomorphism
  algebra.
\newblock {\em Compositio Math.}, 67(1):103--120, 1988.

\bibitem{MR977774}
Frans Oort.
\newblock Endomorphism algebras of abelian varieties.
\newblock In {\em Algebraic geometry and commutative algebra, {V}ol. {II}},
  pages 469--502. Kinokuniya, Tokyo, 1988.

\bibitem{2007.15180}
Kaoru Sano and Takahiro Shibata.
\newblock Zariski density of points with maximal arithmetic degree, 2020.

\bibitem{MR1760844}
Nessim Sibony.
\newblock Dynamique des applications rationnelles de {$\bold P^k$}.
\newblock In {\em Dynamique et g\'{e}om\'{e}trie complexes ({L}yon, 1997)},
  volume~8 of {\em Panor. Synth\`eses}, pages ix--x, xi--xii, 97--185. Soc.
  Math. France, Paris, 1999.

\bibitem{MR1832167}
Bit-Shun Tam.
\newblock A cone-theoretic approach to the spectral theory of positive linear
  operators: the finite-dimensional case.
\newblock {\em Taiwanese J. Math.}, 5(2):207--277, 2001.

\bibitem{MR4048444}
Tuyen~Trung Truong.
\newblock Relative dynamical degrees of correspondences over a field of
  arbitrary characteristic.
\newblock {\em J. Reine Angew. Math.}, 758:139--182, 2020.

\end{thebibliography}
\end{document}